\documentclass[onefignum,onetabnum]{siamart190516}

\usepackage{lipsum}
\usepackage{amsfonts}
\usepackage{graphicx}
\usepackage{epstopdf}
\usepackage{algorithm}
\usepackage{algpseudocode}
\usepackage{fancyvrb}
\usepackage{makeidx}
\usepackage{amsmath}
\usepackage{amssymb}
\usepackage{latexsym,remreset}
\usepackage[toc,page]{appendix}
\usepackage{graphicx}
\usepackage{multirow}
\usepackage{bbm}
\usepackage{amsfonts}
\usepackage{pgfplots}
\usepackage{mathtools}
\usepackage{textcomp}
\usepackage{hyperref}
\usepackage{verbatimbox}
\usepackage{subcaption}
\usepackage{graphicx}
\usepackage{array}
\usepackage{booktabs}
\usepackage{textcomp,array}
\usepackage{enumitem}
\usepackage{stmaryrd}

\DeclareMathOperator*{\argmin}{argmin} 
 
\newcommand{\N}{\mathbb{N}}
\newcommand{\Z}{\mathbb{Z}}
\newcommand{\R}{\mathbb{R}}

\ifpdf
  \DeclareGraphicsExtensions{.eps,.pdf,.png,.jpg}
\else
  \DeclareGraphicsExtensions{.eps}
\fi

\newsiamremark{example}{Example}
\newsiamremark{assumption}{Assumption}
\newsiamremark{remark}{Remark}
\newsiamremark{hypothesis}{Hypothesis}
\crefname{hypothesis}{Hypothesis}{Hypotheses}
\newsiamthm{claim}{Claim}

\headers{Dual Descent ALM and ADMM}{Kaizhao Sun and Xu Andy Sun}

\title{Dual Descent ALM and ADMM
\thanks{Submitted to the editors DATE.
\funding{To be added}}
}

\author{
Kaizhao Sun
\thanks{Decision Intelligence Lab, DAMO Academy, Alibaba Group (U.S.) Inc., Bellevue, WA USA
  (\email{kaizhao.s@alibaba-inc.com}).}
\and 
Xu Andy Sun
\thanks{Sloan School of Management, Massachusetts Institute of Technology, Cambridge, MA, USA
  (\email{sunx@mit.edu}).}
}

\usepackage{amsopn}

\ifpdf
\hypersetup{
  pdftitle={Dual Descent ALM and ADMM},
  pdfauthor={Kaizhao Sun and Xu Andy Sun}
}
\fi

\begin{document}

\maketitle

% REQUIRED
\begin{abstract}
Classical primal-dual algorithms attempt to solve $\max_{\mu}\min_{x} \mathcal{L}(x,\mu)$ by alternatively minimizing over the primal variable $x$ through primal descent and maximizing the dual variable $\mu$ through dual ascent. However, when $\mathcal{L}(x,\mu)$ is highly nonconvex with complex constraints in $x$, the minimization over $x$ may not achieve global optimality, and hence the dual ascent step loses its valid intuition. This observation motivates us to propose a new class of primal-dual algorithms for nonconvex constrained optimization with the key feature to reverse dual ascent to a conceptually new \emph{dual descent}, in a sense, elevating the dual variable to the same status as the primal variable. Surprisingly, this new dual scheme achieves some best iteration complexities for solving nonconvex optimization problems. In particular, when the dual descent step is scaled by a fractional constant, we name it scaled dual descent (SDD), otherwise, unscaled dual descent (UDD). For nonconvex multiblock optimization with nonlinear equality constraints, we propose SDD-ADMM and show that it finds an $\epsilon$-stationary solution in $\mathcal{O}(\epsilon^{-4})$ iterations. The complexity is further improved to $\mathcal{O}(\epsilon^{-3})$ and $\mathcal{O}(\epsilon^{-2})$ under proper conditions. We also propose UDD-ALM, combining UDD with ALM, for weakly convex minimization over affine constraints. We show that UDD-ALM finds an $\epsilon$-stationary solution in $\mathcal{O}(\epsilon^{-2})$ iterations. These complexity bounds for both algorithms either achieve or improve the best-known results in the ADMM and ALM literature. Moreover, SDD-ADMM addresses a long-standing limitation of existing ADMM frameworks. 
\end{abstract}

% REQUIRED
\begin{keywords}
  Augmented Lagrangian Method, Alternating Direction Method of Multipliers
\end{keywords}

% REQUIRED
\begin{AMS}
	65K05,  % Numerical mathematical programming methods
	90C26,  % Nonconvex programming, global optimization
  	90C30,  % Nonlinear programming
  	90C46   % Optimality conditions and duality in mathematical programming
\end{AMS}

\section{Introduction}\label{sec: introduction}
In this paper, we consider the following problem:
\begin{align}\label{eq: mb-nlp}
	\min_{x\in \R^n}  \left\{f(x) + \sum_{i=1}^p g_i(x_i)~\Big|~ \sum_{i=1}^p h_i(x_i) = 0\right\},
\end{align}
where the variable $x\in \R^n$ has the block-coordinate form $x = [x_1^\top,\cdots, x_p^\top]^\top$ with each $x_i \in \R^{n_i}$ for $i\in[p]:=\{1,2,\dots,p\}$ and $\sum_{i=1}^p n_i = n$. We assume $f:\R^n \rightarrow \R$ has a Lipschitz gradient, and each $g_i:\R^{n_i}\rightarrow \overline{\R}:= \R \cup \{+\infty\}$ is proper, lower-semicontinuous, and possibly nonconvex; in addition, for each $i\in [p]$, constraints $h_i:\R^{n_i} \rightarrow \R^m$ are continuously differentiable over the domain of $g_i$. Denote $g(x) := \sum_{i=1}^p g_i(x_i)$ and $h(x): = \sum_{i=1}^p h_i(x_i)$.

The augmented Lagrangian method (ALM), which was proposed in the late 1960s \cite{hestenes1969multiplier,powell1967method}, provides a powerful algorithmic framework for constrained optimization problems including \eqref{eq: mb-nlp}. Define the augmented Lagrangian function as 
\begin{align}\label{eq: AL}
	\mathcal{L}_{\rho}(x, \mu) :=  f(x)  + g(x) + \langle \mu, h(x)\rangle + \frac{\rho}{2}\|h(x)\|^2,
\end{align}
where $\mu \in \R^m$ and $\rho >0$. In the $(k+1)$-th iteration, the ALM firstly obtains the primal iterate $x^{k+1}$ by minimizing the augmented Lagrangian function with dual variable $\mu^{k}$ fixed, possibly in an inexact way: 
\begin{align}\label{eq: alm-p}
	x^{k+1} \approx \argmin_{x\in \R^n} 	\mathcal{L}_{\rho}(x, \mu^k),
\end{align}
and then updates the dual variable using primal residuals:
\begin{align}\label{eq: alm-d}
	\mu^{k+1} = \mu^k + \varrho_k h(x^{k+1}),
\end{align}
where $\varrho_k>0$ is a positive dual step size. 

The ALM framework is flexible: $x^{k+1}$ in \eqref{eq: alm-p} is allowed to be (some approximate counterpart of) a global minimum \cite{rockafellar1976augmented}, a local minimum \cite{bertsekas2014constrained}, or just a stationary point \cite{andreani2008augmented}. Another possibility is to update blocks of variables $(x_1, \cdots, x_p)$ in a coordinate fashion, i.e., through a Gauss-Seidel or Jacobi sweep; when $h$ is affine, algorithms of this type are commonly known as the alternating direction method of multipliers (ADMM). The dual update \eqref{eq: alm-d} is motivated by the fact that the augmented Lagrangian dual function 
\begin{align} \label{eq: dual function}
	d(\mu) := \min_{x\in \R^n} \mathcal{L}_{\rho}(x, \mu)
\end{align}
is concave, and $-h(x^{k+1}) \in \partial_{\epsilon} (-d)(\mu^{k})$ with any $x^{k+1}$ such that $\mathcal{L}_{\rho}(x^{k+1}, \mu^k) \leq  d(\mu^k) + \epsilon$. In this case, the update \eqref{eq: alm-d} is essentially maximizing the concave function $d$ using an inexact subgradient of $-d$. We refer to \eqref{eq: alm-d} as a \textit{dual ascent} step. A motivation for this paper is that the classic interpretation of dual ascent of \eqref{eq: alm-d} is not valid anymore if the gap between $d(\mu^k) $ and $\mathcal{L}_{\rho}(x^{k+1}, \mu^k)$ is large or cannot be uniformly bounded over iterations, especially when $x^{k+1}$ is a local minimum, a stationary point, or a coordinate-wise solution of the nonconvex function $\mathcal{L}_{\rho}(\cdot, \mu^k)$. 

{This observation opens up new possibilities for algorithmic design within the augmented Lagrangian framework. Given $\mu^k \in \mathbb{R}^m$, let $x^{k+1}$ represent a coordinate-wise solution of $\mathcal{L}(\cdot, \mu^k)$. Notably, since $x^{k+1}$ does not provide valid zero/first-order information of $d$ at $\mu^k$, the intuition of the dual ascent \eqref{eq: alm-d} is lost. Instead of maximizing $d$, the fact that $\nabla_\mu \mathcal{L}{\rho}(x^{k+1}, \cdot) = h(x^{k+1})$ suggests an alternative approach. By ``minimizing" $\mathcal{L}{\rho}$ with respect to $\mu$ and assuming an approximate stationary point can be attained, it is expected that the primal residual $\|h(x^{k+1})\|$ will be small. To pursue this idea, one might be inclined to employ block-coordinate descent (BCD) algorithms \cite{xu2013block, xu2017globally} for $\mathcal{L}{\rho}(x_1, \cdots, x_p, \mu)$; however, the linearity of the function $\mathcal{L}{\rho}(x, \cdot)$ renders $\mathcal{L}_{\rho}$ potentially unbounded in the dual variable $\mu$. To address this, we introduce a regularized augmented Lagrangian function:
\begin{align} \label{eq: regularized AL}
\mathcal{P}(x, \mu) := \mathcal{L}_{\rho}(x, \mu) + \frac{\omega}{2\rho}\|\mu\|^2,
\end{align}
where we include a quadratic term in $\mu$ with $\omega > 0$. Once $x^{k+1}$ is obtained, for example, through a Gauss-Seidel sweep of proximal gradient updates, we can update $\mu^{k+1}$ using the following formulation:
\begin{align} \label{eq: sdd-intro}
\mu^{k+1} = \argmin_{\mu \in \mathbb{R}^m} \mathcal{P}(x^{k+1}, \mu) + \frac{\tau \omega}{2\rho}\|\mu - \mu^k\|^2 = \underbrace{ \frac{\tau }{1+\tau}}_{\text{scaled}} \underbrace{ \left( \mu^k - \tau^{-1}\omega^{-1} \rho h(x^{k+1}) \right) }_{\text{dual descent}},
\end{align}
where $\tau > 0$ and this update is referred to as the \textit{scaled dual descent} (SDD).}

The above update ensures 1) sufficient descent and lower-boundedness of $\mathcal{P}$ and 2) boundedness of the sequence $\{\mu^{k+1}\}_{k\in \N}$, which are critical for the convergence rate analysis of ALM-based algorithms. In particular, we show that, with a near-feasible initialization, the SDD update gives an $\epsilon$-stationary solution of \eqref{eq: mb-nlp} in $\mathcal{O}(\epsilon^{-4})$ iterations, which can be further improved to $\mathcal{O}(\epsilon^{-3})$ {and  $\mathcal{O}(\epsilon^{-2})$} under additional verifiable assumptions. Inspired by a comment from a referee, we have made an intriguing observation. By using a different proximal center $\hat{\mu}^{k}$ instead of $\mu^k$ in \eqref{eq: sdd-intro}, defined as:
\begin{align}\label{eq: mu_hat}
\hat{\mu}^k := \mu^k + \frac{\rho}{\omega \tau}{h(x^{k+1})},
\end{align}
we find that the dual variable vanishes in all iterations when initialized with zeros. This realization highlights the versatility of the SDD framework, as it not only introduces a novel class of dual updates but also encompasses the classic penalty method when combined with a traditional dual ascent step \eqref{eq: mu_hat}. Consequently, we provide a unified convergence analysis for both SDD and the penalty method, with the complexity results for the latter being novel contributions to the literature.

A natural question then arises: what will happen if we simply perform an \textit{unscaled dual descent} (UDD) update, i.e., 
\begin{align}\label{eq: udd-intro}
	\mu^{k+1} = \mu^k - \varrho \nabla_\mu \mathcal{L}_{\rho} (x^{k+1}, \mu^k) = \mu^k - \varrho h(x^{k+1}),
\end{align}
where $\varrho >0$ is a fixed dual step size. The analysis of UDD presents a main technical challenge in establishing the boundedness of the dual variable to prevent the augmented Lagrangian function from becoming unbounded from below. In this paper, we provide some positive theoretical results and preliminary empirical observations for the UDD update. From a theoretical perspective, we demonstrate that when some regularity condition holds at the primal limit point, regardless of the choice of $\varrho > 0$, the dual sequence has a bounded subsequence and hence the augmented Lagrangian function is lower bounded; as a result, the UDD update finds an $\epsilon$-stationary solution in $\mathcal{O}(\epsilon^{-2})$ iterations. On the empirical side, we observe that UDD converges on a simple consensus problem when the step size $\varrho$ is close to zero. In this scenario, the UDD update \eqref{eq: udd-intro} describes the limiting behavior of the SDD update \eqref{eq: sdd-intro} with $\tau$ converging to $+\infty$ and $\omega$ remaining constant. Essentially, the empirical convergence of UDD could be attributed to the penalty method, where the update of dual variable is relatively negligible.

To our knowledge, references \cite{jian2021qcqp,jian2020monotone} are the only two works that use a related idea of dual descent. The authors consider a special case of \eqref{eq: mb-nlp} with two blocks of variables with continuously differentiable coupling constraints. In each iteration, the proposed algorithm first solves two quadratic programs (QPs), and then uses the solutions to determine the primal and dual descent directions of the augmented Lagrangian function. Assuming boundedness of dual variables, convergence to a stationary point is proved with an iteration complexity of $\mathcal{O}(\epsilon^{-2})$ QP oracles. The proposed dual descent framework in this paper is different from \cite{jian2021qcqp,jian2020monotone} in several nontrivial perspectives. We summarize our contributions in the next subsection.

\subsection{Contributions}
We summarize our contributions as follows. We introduce SDD within the augmented Lagrangian framework to solve nonlinear constrained nonconvex problem \eqref{eq: mb-nlp}. In iteration $k+1$, we obtain $x^{k+1}$ through a Gauss-Seidel sweep of proximal gradient updates, and then update the dual variable $\mu^{k+1}$ via an SDD step. We call the resulting algorithm SDD-ADMM when $p>1$, and SDD-ALM when $p=1$. In contrast to most existing ADMM and ALM works considering only affine constraints (see Section \ref{sec: literature} for a detailed review), the proposed SDD-ADMM and SDD-ALM are able to handle nonlinear smooth coupling constraints of the form $\sum_{i=1}^p h_i(x_i)=0$, and therefore are applicable to a broader class of problems. 

In addition to being able to handle nonlinear constraints, SDD-ADMM ($p>1$) achieves better iteration complexities under a more general setting. Compared to existing multi-block nonconvex ADMM works \cite{gonccalves2017convergence,hong2016convergence,jiang2019structured,melo2017iteration,melo2017jacobi,wang2015global}, we do not impose restrictive assumptions on problem data (see Section \ref{sec: admm lit}), {and we show that SDD-ADMM obtains an $\epsilon$-stationary solution in $\mathcal{O}(\epsilon^{-4})$ iterations, which can be further improved to $\mathcal{O}(\epsilon^{-3})$ or $\mathcal{O}(\epsilon^{-2})$ under suitable conditions.  Our $\mathcal{O}(\epsilon^{-4})$ and $\mathcal{O}(\epsilon^{-3})$ estimates significantly improve the existing $\mathcal{O}(\epsilon^{-6})$ \cite{jiang2019structured} and $\mathcal{O}(\epsilon^{-4})$ \cite{sun2019two} complexities, respectively, and our $\mathcal{O}(\epsilon^{-2})$ estimate complements the above mentioned references.} Moreover, our iteration complexities are measured by first-order oracles, i.e., gradient oracles of $f$ and $h$ and proximal oracles of $g_i$’s, which are in general more tractable than the subproblem oracles considered in \cite{jiang2019structured,sun2019two}.

For SDD-ALM ($p=1$), our iteration complexities slightly improve the best-known results in \cite{lin2019inexact} (without a technical assumption) and  \cite{li2021rate} (with a technical assumption) by getting rid of the logarithmic dependency on $\epsilon^{-1}$. Another feature of SDD-ALM is that the algorithm is single-looped, which might be preferable ove double- or triple-looped ALM and penalty methods \cite{li2021rate, lin2019inexact, Sahin2019alm, xie2021complexity} from an implementation point of view, i.e., the technicality of choosing the inner-loop stopping criteria is avoided. In addition, our convergence analysis and complexity estimates also apply to an interesting single-looped first-order penalty method.

To further understand the behavior of dual descent, we introduce UDD within the augmented Lagrangian framework and name the resulting algorithm UDD-ALM. We first investigate UDD-ALM for weakly convex minimization with affine constraints and show that when a certain regularity condition holds at the primal limit point, UDD-ALM finds an $\epsilon$-stationary solution in $\mathcal{O}(\epsilon^{-2})$ iterations. UDD-ALM is single-looped and our iteration complexity is again measured by first-order oracles of $f$ and $g$. We do not restrict $g$ to be an indicator function of a box or polyhedron, and hence our result complements those in \cite{zhang2020proximal,zhang2020global}. Finally, we extend the analysis of UDD-ALM to handle nonconvex $g$ and nonlinear constraints $h$ by assuming a novel descent oracle of each proximal augmented Lagrangian relaxation over the domain of $g$. We would like to acknowledge that there is still a need for a deeper understanding of the behavior of UDD, particularly concerning the implicit impact of the dual step size on the regularity assumption we imposed on the primal limit point. We do not claim or advocate the superiority of UDD over existing algorithms, but simply share our current theoretical understanding and empirical observations on this counter-intuitive approach. Our hope is that both SDD and UDD can serve as catalysts for inspiring further algorithmic developments that go beyond traditional approaches.

\subsection{Notations}
We denote the set of positive integers up to $p$ by $[p]$, the set of nonnegative integers by $\N$, the set of real numbers by $\R$, the set of extended real line by $\overline{\R}:=\R\cup \{+\infty\}$, and the $n$-dimensional real Euclidean space by $\R^n$. For $x, y\in \R^n$, the inner product of $x$ and $y$ is denoted by $\langle x, y \rangle$, and the Euclidean norm of $x$ is denoted by $\|x\|$; for $A \in \R^{m\times n}$, {$\|A\|$ and $\sigma_{\min}(A)$ denote the largest and smallest singular value of $A$, respectively. For $X \subseteq \R^n$, we use $\delta_X$ to denote the $0/\infty$-indicator function of $X$.}
Denote $\mathrm{dist}(x, x) := \inf_{y\in X} \|y-x\|$. When $x \in \R^n$ has the block-coordinate form $[x_1^\top,\cdots, x_p^\top]^\top$ with $x_i \in \R^{n_i}$ for $i\in[p]$ and $\sum_{i=1}^p n_i = n$, we denote $x_{\leq i} = [x_1^\top, \cdots, x_{i}^\top]^\top$ and $x_{\geq i} = [x_i^\top, \cdots, x_{p}^\top]^\top$ for $i\in [p]$, and similarly for $x_{<i} $ and $x_{>i}$. We use the following notations for $x$ when it is necessary to distinguish variable blocks (specifically when $x$ is an argument of a function): $(x_{\leq i}, x_{> i})$, $(x_{<i}, x_i, x_{>i})$, $(x_{<i}, x_{\geq i})$, {or $(x_{\neq i}, x_i)$}. Finally, we adopt the following notations for complexity analysis. Let $\varepsilon >0$ and $K >0$. We write $K = \mathcal{O}(\varepsilon)$ if $K \leq B\varepsilon$ for some $0< B< +\infty$, and $K = \Theta(\varepsilon)$ if $b\varepsilon \leq K \leq B\varepsilon$ for some $0 <b < B< +\infty$.
\subsection{Organization}
The rest of the paper is organized as follows. Section \ref{sec: literature} reviews related works in ALM and ADMM. In Section \ref{sec: sdd}, we introduce the SDD-ADMM algorithm, present its convergence analysis {as well as an adaptive version}, and discuss its connection with existing algorithms. In Section \ref{sec: udd}, we establish the convergence of UDD-ALM under the setting where $p=1$, $h(x)$ is affine, and $g$ is convex, and further extend the analysis to handle nonconvex $g$ and nonlinear $h$.  {We present some numerical experiments in Section \ref{section: examples} and finally conclude this paper in Section \ref{sec: conclusion}.}

\section{Related Works}\label{sec: literature}
This section reviews the literature of ALM and ADMM.
\subsection{ALM}
 The asymptotic convergence and convergence rate of ALM have been extensively studied for convex programs by \cite{lanMonteiro2016alm,rockafellar1973multiplier,rockafellar1976augmented} and smooth nonlinear programs \cite{andreani2008augmented,bertsekas2014constrained}. In this subsection, we review some recent developments in ALM-based algorithms applied to nonconvex problems of the form:
\begin{align}\label{eq: alm literature}
	\min_{x\in \R^n} \left\{ F(x)~|~ h_E(x) = 0, h_I(x) \leq 0\right\}.
\end{align}
Often the objective $F$ is assumed to admit a composite form $f+g$, where $f$ has Lipschitz gradient and $g$ is a nonsmooth convex function.
\subsubsection{Convex Constraints}
Works \cite{hajinezhad2019perturbed,hong2016decomposing,melo2020almfullstep,melo2020iteration,sun2021algorithms,zeng2021moreau, zhang2020proximal,zhang2020global} consider affine constraints, i.e., $h_E(x) = Ax-b=0$, while inequality constraints $h_I(x)\leq 0$ are not present. For a special case with $g=0$, Hong \cite{hong2016decomposing} proposed a proximal primal-dual algorithm (prox-PDA) that finds an approximate stationary point in $\mathcal{O}(\epsilon^{-2})$ iterations, where $\epsilon>0$ measures both first-order stationarity and feasibility (``$\epsilon$-stationary point" hereafter). When $g$ is a compactly supported convex function, Hajinezhad and Hong \cite{hajinezhad2019perturbed} proposed a perturbed prox-PDA that achieves an iteration complexity of $\mathcal{O}(\epsilon^{-4})$. Zeng et al. \cite{zeng2021moreau} proposed a Moreau Envelope ALM (MEAL) for handling a general weakly convex objective function $F$, which achieves the $\mathcal{O}(\epsilon^{-2})$ iteration complexity. Authors of \cite{sun2021algorithms} proposed two variants of ALM with $\mathcal{O}(\epsilon^{-2})$ iteration estimates when $F$ is a difference-of-convex (DC) function.

In contrast to previously mentioned works where iteration complexities are measured by the number of times a (proximal) augmented Lagrangian relaxation is solved, the following works study iteration complexities in terms of \textit{first-order oracles}, i.e., the number of proximal gradient steps. Melo et al. \cite{melo2020iteration} applied an accelerated composite gradient method (ACG) \cite{beck2009fast} to solve each proximal ALM subproblem and showed that an $\epsilon$-stationary point can be found in $\tilde{\mathcal{O}}(\epsilon^{-3})$\footnote{The notation $\tilde{\mathcal{O}}$ hides logarithmic dependence on $\epsilon^{-1}$.} ACG iterations, which can be further reduced to $\tilde{\mathcal{O}}(\epsilon^{-5/2})$ with mildly stronger assumptions. Later, Melo and Monteiro \cite{melo2020almfullstep} embedded this inner acceleration scheme in a proximal ALM with full dual multiplier update and derived an iteration complexity of $\tilde{\mathcal{O}}(\epsilon^{-3})$ ACG iterations. Zhang and Luo proposed a single-looped proximal ALM and established an $\mathcal{O}(\epsilon^{-2})$ iteration estimate when $g$ is an indicator function of a hypercube \cite{zhang2020proximal} or a polyhedron \cite{zhang2020global}.

Works \cite{kong2020iteration,LiXu2020almv2} study the setting where convex nonlinear constraints $h_I(x)\leq 0$ are explicitly present. When $F$ is weakly convex and $h_E(x)$ is affine, Li and Xu \cite{LiXu2020almv2} combined an inexact ALM and a quadratic penalty method, and they showed that an $\epsilon$-stationary point can be found in $\tilde{\mathcal{O}}(\epsilon^{-5/2})$ adaptive accelerated proximal gradient (APG) steps. Kong et al. \cite{kong2020iteration} studied a more general setting: convex nonlinear constraints take the form $-h_I(x) \in \mathcal{K}$ where $\mathcal{K}$ is a closed convex cone. Under the same inner acceleration scheme as in \cite{melo2020almfullstep}, they showed that the proposed proximal ALM, named NL-IAPIAL, finds an $\epsilon$-stationary solution in $\tilde{\mathcal{O}}(\epsilon^{-3})$ ACG iterations.

\subsubsection{Nonconvex Constraints}
Works \cite{li2021rate,lin2019inexact, Sahin2019alm, xie2021complexity} consider nonlinear and nonconvex constraints. Sahin et al. \cite{Sahin2019alm} studied problem \eqref{eq: alm literature} with only equality constraints $h_E(x)=0$, and proposed a double-looped inexact ALM (iALM), where the augmented Lagrangian relaxation is solved by the accelerated gradient method in \cite{ghadimi2016accelerated}, and then the dual variable is updated with a small step size, which ensures that the sequence of dual variables is uniformly bounded. Assuming a technical \textit{regularity condition} that provides a convenient workaround to control primal infeasibility using dual infeasibility, the proposed iALM achieves an $\tilde{\mathcal{O}}(\epsilon^{-4})$ iteration complexity\footnote{The complexity presented in \cite{Sahin2019alm} is claimed to be wrong and corrected to $\tilde{\mathcal{O}}(\epsilon^{-4})$ by \cite{li2021rate}.}. Assuming the same regularity condition, Li et al. \cite{li2021rate} later improved the iteration complexity to $\tilde{\mathcal{O}}(\epsilon^{-3})$, which is achieved through a triple-looped iALM.

For problems where the aforementioned regularity condition is not satisfied, the inexact proximal point penalty (iPPP) method proposed by Lin et al. \cite{lin2019inexact} finds an $\epsilon$-stationary solution in $\tilde{\mathcal{O}}(\epsilon^{-4})$ adaptive APG steps under the requirement that the initial point is feasible. Xie and Wright \cite{xie2021complexity} proposed a proximal ALM that finds an $\epsilon$-stationary solution in $\mathcal{O}(\epsilon^{-5.5})$ Newton-CG iterations.
	
\subsection{ADMM}\label{sec: admm lit}
The alternating direction method of multipliers (ADMM) was proposed in the mid-1970s \cite{gabay1976dual,glowinski1975approximation}, while the underlying idea has deep roots in maximal monotone operator theory \cite{eckstein1992douglas} and numerical methods for partial differential equations \cite{douglas1956numerical,peaceman1955numerical}. Commonly regarded as a variant of ALM, ADMM solves the augmented Lagrangian relaxation by alternately optimizing through blocks of variables and in this way subproblems become decoupled. Such a feature has gained ADMM considerable attentions in distributed optimization \cite{boyd2011distributed,makhdoumi_broadcast-based_2014,shi_linear_2014}. The convergence of ADMM with two block variables is proved for convex optimization problems \cite{eckstein1992douglas,gabay2024applications,gabay1976dual,glowinski1975approximation} and a convergence rate of $\mathcal{O}(1/K)$ is established \cite{he20121,he2015non,monteiro2013iteration}, where $K$ is the iteration index. See also \cite{chen2016direct,davis2017three,gonccalves2016extending,hong2017linear,lin2018global} on convex multi-block ADMM.

In recent years, researchers have extended the ADMM framework to solve nonconvex multi-block problem \eqref{eq: mb-nlp}, where $h_i(x_i) = A_ix_i$ is affine for all $i\in [p]$ \cite{gonccalves2017convergence,hong2016convergence,jiang2019structured,melo2017iteration,melo2017jacobi,wang2015global}. 
The asymptotical convergence and an iteration complexity of $\mathcal{O}(\epsilon^{-2})$ are established based on two crucial conditions on the problem data: (a) $g_p = 0$ and (b) the column space of $A_p$ contains the column space of the concatenated matrix $[A_1, \cdots, A_{p-1}]$.Condition (a) provides a way to control dual iterates by primal iterates, while condition (b) is required for ADMM to locate a feasible solution in the limit. See also \cite[Table 1]{sun2019two} for a summary of other assumptions. These two assumptions are almost necessary for the convergence of nonconvex ADMM. Namely, when either one of the two assumptions fails to hold, divergent examples have been found. 

There are also several works investigating the convergence of nonconvex ADMM without conditions (a) and (b). In particular, Jiang et al. \cite{jiang2019structured} proposed to run ADMM on a penalty relaxation of \eqref{eq: mb-nlp}. K. Sun and XA. Sun \cite{sun2019two,sun2021two} proposed a two-level framework that embeds a structured three-block ADMM inside an ALM. For weakly convex minimization over affine constraints, works \cite{zeng2021moreau,zhang2020proximal} demonstrate two ADMM variants that do not require assumptions (a) or (b).

Zhu et al. \cite{zhu2020first} considered nonlinear coupling constraints of the form $h(x_1) + Bx_2=0$. Assuming condition (a) and a straightforward extension of condition (b), i.e., the range of the nonlinear mapping $h$ belongs to the column space of the matrix $B$, the authors derived the $\mathcal{O}(\epsilon^{-2})$ iteration complexity.

\section{Scaled Dual Descent ADMM}\label{sec: sdd}

\subsection{Assumptions and Stationarity}
In this subsection, we formally state our assumptions on the problem data and define stationarity for problem \eqref{eq: mb-nlp}.
\begin{assumption}~\label{assumption: mb problem data}
	\begin{enumerate}
		\item For $i\in [p]$, the function $g_i:\R^{n_i} \rightarrow \overline{\R}$ is proper and lower-semicontinuous with effective domain denoted by $X_i = \mathrm{dom}~g_i := \{x\in \R^n~|~g_i(x) < +\infty\}.$
		% \begin{align}
		% 	X_i = \mathrm{dom}~g_i := \{x\in \R^n~|~g_i(x) < +\infty\}.
		% \end{align}
		Moreover, the proximal oracle of $g_i$ is available, i.e., given $z_i \in \R^{n_i}$ and a sufficiently large constant $\eta >0$, we can solve the following problem
		\begin{align}\label{eq: prox-oracle-g}
			\min_{x_i\in \R^{n_i}} g_i(x_i) + \frac{\eta}{2}\|x_i-z_i\|^2.
		\end{align}
		Denote $g(x) = \sum_{i=1}^p g_i(x_i)$ and $X = \prod_{i=1}^p X_i$.
		\item The function $f:\R^n \rightarrow \R$ has Lipschitz gradient over $X$, i.e., there exists a positive constant $L_{f}$ such that $\|\nabla f(x)- \nabla f(z) \|	\leq L_{f} \|x-z\|$
		% \begin{align}
		% 	\|\nabla f(x)- \nabla f(z) \|	\leq L_{f} \|x-z\|
		% \end{align}
 		for any $x, z\in X$. 
		\item The mapping $h:\R^n \rightarrow \R^m$ is given by $h(x) = \sum_{i=1}^p h_i(x_i)$, where $h_i:\R^{n_i}\rightarrow \R^m$ is continuously differentiable over $X_i$, and there exist positive constants $M_{h_i}$, {$K_{h_i}$}, $J_{h_i}$, and $L_{h_i}$ such that for all $i\in [p]$ and $x_i, z_i\in X_i$, we have 
		\begin{subequations}\label{eq: mp h const}
		\begin{alignat}{2}
			 & \max_{x_i\in X_i}\|h_i(x_i)\|\leq  M_{h_i}, 	 \quad  && \|h_i(x_i)-h_i(z_i)\| \leq  K_{h_i}\|x_i-z_i\|,   \\
			 & \max_{x_i\in X_i}\|\nabla h_i(x_i)\| \leq  J_{h_i},  && \|\nabla h_i(x_i)-\nabla h_i(z_i)\| \leq   L_{h_i}\|x_i-z_i\|, 
		\end{alignat}
		\end{subequations}
		where $\nabla h_i(x_i) = [\nabla h_{i1}(x_i), \cdots, \nabla h_{im}(x_i)]\in \R^{n_i \times m}$, and $\|\cdot\|$ denotes the Euclidean norm for vectors or the induced norm for matrices.
		\item {The following constants are finite:
		\begin{subequations}\label{eq: upper and lower bound on P}
		\begin{align}
			\overline{\mathcal{P}} :=& \sup_{x\in X} f(x) + \sum_{i=1}^p g_i(x_i) < +\infty, \label{eq: p_ub}\ \text{and}\\ 
			\underline{\mathcal{P}}:=& \inf_{x\in X} f(x) + \sum_{i=1}^p g_i(x_i)  > -\infty.	\label{eq: p_lb}
		\end{align}	
		\end{subequations}
		}
%		\item Problem \eqref{eq: mb-nlp} is feasible, i.e., $X\cap\{x\in\R^n~|~h(x)=0\}\neq \emptyset$, and 
%		\begin{align}\label{eq: p_lb}
%			\underline{\mathcal{P}} :=  \min_{x\in X} f(x) + \sum_{i=1}^p g_i(x_i)  > -\infty.	
%		\end{align}
	\end{enumerate}
\end{assumption}
For ease of later presentation, let
\begin{align}\label{eq: mb-h-constants}
	M_h := \sum_{i=1}^p M_{h_i} , \ K_h :=\max_{i\in [p]} \{K_{h_i}\}, \ J_h :=  \max_{i\in [p]} \{J_{h_i}\}, \ \text{and} \ L_h :=\max_{i\in [p]} \{L_{h_i}\}.
\end{align}
%\begin{assumption}~\label{assumption: feasible initial pt}
%	A feasible {point} $x^0 \in X\cap\{x\in\R^n~|~h(x)=0\}$ is known.
%\end{assumption}
\begin{remark} 
	We give some remarks regarding the above assumptions.
\begin{enumerate}
	\item We allow $g_i$ to be nonconvex as long as its proximal oracle is available. 
	 Without loss of generality, we may assume any $\eta\geq L_f$ suffices to carry out the minimization in \eqref{eq: prox-oracle-g} exactly.
	 \item {Any $g_i$ of the form $\delta_{X_i} + \tilde{g}_{i}$ where $X_i$ is a compact set and $\tilde{g}_{i}$ is continuous over $X_i$ ensures that constants in \eqref{eq: mp h const} and \eqref{eq: upper and lower bound on P} are well-defined. However, we do not explicitly require the compactness of $X_i$'s because many nonconvex $g_i$'s such as SCAD, MCP, and Capped-$\ell_1$ are defined over $\R^{n_i}$ and uniformly bounded from above, and nonlinear mappings including the sine, cosine, arctangent, and sigmoid functions  can ensure \eqref{eq: mp h const} without $X_i$'s being compact. In addition, it is possible to further relax condition \eqref{eq: p_ub}; see Remark \ref{remark: feasibility}.}
	\item Suppose each $h_{ij}:\R^n \rightarrow \R$ satisfies that $\|\nabla h_{ij}(x_i)-\nabla h_{ij}(z_i)\|\leq L_{h_{ij}}\|x_i-z_i\|$ and $\|\nabla h_{ij}(x_i)\| \leq J_{h_{ij}}$ for all $x_i, z_i\in X_i$, then we can obtain the following estimates: $L_{h_i} = \sqrt{m}\max_{j\in [m]} L_{h_{ij}}$ and $J_{h_i} = K_{h_i} = \sqrt{m} \max_{j\in [m]} J_{h_{ij}}$.
\end{enumerate}
\end{remark}

We define approximate stationarity as follows. 
\begin{definition}[Approximate Stationary Point]
	Given $\epsilon>0$, we say $x\in X$ is an $\epsilon$-stationary point of problem \eqref{eq: mb-nlp} if there exists $\lambda \in \R^m$ such that
	\begin{align*}
		 \max_{i\in [p]} \Big \{ \mathrm{dist}\Big( -\nabla_i f(x)- \nabla h_i(x_i) \lambda , \partial g_i(x_i) \Big) \Big\} \leq \epsilon, \ \text{and} \ \|h(x)\|\leq \epsilon,
	\end{align*}
	where $\partial g_i(x_i)$ denotes the general subdifferential of $g_i$ at $x_i \in X_i$ \cite[Definition 8.3]{rockafellar2009variational}.
\end{definition}

\subsection{The Proposed Algorithm}
Recall the augmented Lagrangian function in \eqref{eq: AL}. We also separate out the smooth components in $\mathcal{L}_{\rho}(x, \mu)$ as
\begin{align}\label{eq: smooth AL}
	\mathcal{K}_{\rho}(x, \mu) :=  \mathcal{L}_{\rho}(x, \mu)- \sum_{i=1}^p g_i(x_i) = f(x)  + \langle \mu, h(x)\rangle + \frac{\rho}{2}\|h(x)\|^2.
\end{align}
Notice that for $i\in [p]$, $\nabla_{x_i} \mathcal{K}_{\rho}(x, \mu)	= \nabla_i f(x) + \nabla h_i(x_i) (\mu + \rho h(x))$. It can be verified that $\nabla_{x_i} \mathcal{K}_{\rho}(x, \mu)$ is Lipschitz. The calculation is straightforward and hence omitted.
\begin{lemma}\label{eq: lip-grad-k}
	For any $i\in [p]$, $x_i, z_i \in X_i$, and fixed $x_j \in X_j$ with $j\neq i$, we have	
	\begin{align*}
		\|\nabla_{x_i} \mathcal{K}_{\rho}(x_{<i}, x_i, x_{>i}, \mu) - \nabla_{x_i} \mathcal{K}_{\rho}(x_{<i}, z_i, x_{>i}, \mu) \| \leq  \mathrm{Lip}(\mu, \rho) \|x_i-z_i\|,
	\end{align*}
	where
	\begin{align}\label{eq: admm-lip}
	\mathrm{Lip}(\mu, \rho):= L_f + \|\mu\| L_h + \rho (J_h K_h + M_h L_h) .
\end{align}
\end{lemma}

Lemma \ref{eq: lip-grad-k} allows us to update each $x_i$ via a single proximal gradient step. The scaled dual descent ADMM  (SDD-ADMM) is presented in Algorithm \ref{alg: SDD-ADMM}.
\begin{algorithm}[H]
	\caption{: \texttt{SDD-ADMM}} \label{alg: SDD-ADMM}
	\begin{algorithmic}[1]
		\State \textbf{Input} {$x^0 \in X, \; \rho>0, \; \omega\ge 4, \; \theta>1, \; \tau\ge 0$;}\label{alg: admm-init} 
		\State \textbf{initialize} {$\mu^0 = 0 \in \R^m$;}
		\For{$k=0,1,2,\cdots$}
		\For{$i=1,2,\cdots, p$}
		\State update $x_i^{k+1}$ through a proximal gradient step:
		\begin{align}\label{eq: sdd-admm-x}
			x^{k+1}_{i} = \argmin_{x_i\in \R^{n_i}} g_i(x_i) + \langle \nabla_{x_i} \mathcal{K}_{\rho}(x^{k+1}_{<i},x^k_{\geq i}, \mu^k), x_i-x_i^k \rangle + \frac{\theta \mathrm{Lip}(\mu^k,\rho)}{2}\|x_i-x_i^k\|^2;
		\end{align}
		\EndFor
		\State update $\mu^{k+1}$ by
		\begin{align}\label{eq: sdd-admm-mu}
			\mu^{k+1}  = & {\argmin_{\mu \in \R^m} \mathcal{P}(x^{k+1}, \mu) + \frac{\tau \omega}{2\rho}\|\mu - \mu^k\|^2 =}  \frac{1}{1+\tau} \left( \tau \mu^k - \omega^{-1}  \rho h(x^{k+1}) \right); 
		 \end{align}
		\EndFor
	\end{algorithmic}
\end{algorithm}
\begin{remark}
	We {give} some remarks on Algorithm \ref{alg: SDD-ADMM}.
\begin{enumerate}
	\item {We update $x^{k+1}_i$'s through $p$ proximal gradient updates, {and then obtain $\mu^{k+1}$ by minimizing $\mathcal{P}(x^{k+1}, \mu)$ plus a proximal term} as in \eqref{eq: sdd-admm-mu}. Due to the strong convexity of $\mathcal{P}$ in $\mu$, $\tau$ is allowed to be $0$.}
	\item {If we perturb the proximal term by a specific \textit{linear} term to cancel out the inner product in $\mathcal{P}(x^{k+1}, \mu)$, and update $\mu^{k+1}$ as
%	 \begin{align} \label{eq: penalty-mu}
%	 	\mu^{k+1} = 
%	 	\begin{cases}
%	 		\argmin_{\mu \in \R^m} \mathcal{P}(x^{k+1}, \mu) +  \frac{\tau \omega}{2\rho}\|\mu - (\mu^k + \frac{\rho}{\omega \tau}h(x^{k+1}))\|^2 &\text{~if~} \tau > 0 ,\\
%	 		\argmin_{\mu \in \R^m} \mathcal{P}(x^{k+1}, \mu) + \langle -h(x^{k+1}), \mu \rangle &\text{~if~} \tau = 0,
%	 	\end{cases}
%%	 	\bar{\mu}^{k} = \mu^k + \frac{\rho}{\tau \omega}h(x^{k+1}) 
%	 \end{align}
	\begin{align}\label{eq: penalty-mu}
		\mu^{k+1} = \argmin_{\mu \in \R^m} \mathcal{P}(x^{k+1}, \mu) +  \frac{\tau \omega}{2\rho}\|\mu - \mu^k\|^2 + \langle -h(x^{k+1}), \mu \rangle, 
	\end{align}
	then, since we initialize with $\mu^0 = 0$, \eqref{eq: penalty-mu} recovers the penalty method, i.e., $\mu^{k+1} = 0$ for all $k\in \N$. When $\tau >0$, 
	\eqref{eq: penalty-mu} is also equivalent to the SDD update with a different proximal center: 
	\begin{align}\label{eq: penalty-mu-prox}
		\mu^{k+1} = \argmin_{\mu \in \R^m} \mathcal{P}(x^{k+1}, \mu) +  \frac{\tau \omega}{2\rho}\|\mu - (\mu^k + \frac{\rho}{\omega \tau}h(x^{k+1}))\|^2.
	\end{align}
	It is interesting that the new proximal center is a dual ascent iterate, which recovers the penalty method in combination with the SDD framework.
	} 
%	\todo{[TODO: check this single-looped first-order penalty method has not been studied and has better complexity results over existing ones.]}
	\item The lower bound of the parameter $\omega$ is chosen to be 4 mainly for the ease of analysis, e.g., in Lemma \ref{lemma: admm-bound-primal-res}. Other values are possible as well.
	\end{enumerate}
\end{remark}

When $p=1$, we call the algorithm SDD-ALM. One key observation is that, although $x_i$'s are updated in a Gauss-Seidel fashion in Algorithm \ref{alg: SDD-ADMM}, we can also apply a Jacobi-type update, i.e., replace $\nabla_{x_i} \mathcal{K}_{\rho}(x^{k+1}_{<i},x^k_{\geq i}, \mu^k)$ by $\nabla_{x_i} \mathcal{K}_{\rho}(x^k, \mu^k)$ in \eqref{eq: sdd-admm-x} and then solve the $p$ subproblems in parallel. This is a special case when we treat $x = [x_1^\top, \cdots, x_p^\top]^\top$ as a single block variable and apply SDD-ALM, where decomposition is achieved within this single block update and assembles a Jacobi sweep.
Such a feature might be favored in a distributed optimization setting. We present the convergence of SDD-ADMM in the form of  Algorithm \ref{alg: SDD-ADMM} in the next subsection.

\subsection{Convergence Analysis}
In this subsection, {we analyze the convergence of Algorithm \ref{alg: SDD-ADMM} under dual updates \eqref{eq: sdd-admm-mu} and \eqref{eq: penalty-mu} in a unified framework.} We first show that when the initial point is almost feasible, SDD-ADMM finds an $\epsilon$-stationary solution in $\mathcal{O}(\epsilon^{-4})$ iterations. Then under an additional assumption regarding $h$ and $g$, we can further improve the complexity by one or even two orders of magnitude to $\mathcal{O}(\epsilon^{-3})$ and $\mathcal{O}(\epsilon^{-2})$, respectively.

% {While our analysis covers both the SDD \eqref{eq: sdd-admm-mu} and the penalty method \eqref{eq: penalty-mu}, technical claims below are stated in the context of \eqref{eq: sdd-admm-mu}, where the dual variable $\mu$ is presented. This is because our algorithm is motivated by minimizing the augmented Lagrangian function in both primal and dual variables in the first place. In each proof, we accommodate to \eqref{eq: penalty-mu} when needed.}
{
Although our analysis encompasses both the SDD update \eqref{eq: sdd-admm-mu} and the penalty method update \eqref{eq: penalty-mu}, the technical claims presented below are stated within the context of the SDD update \eqref{eq: sdd-admm-mu}, which explicitly involves the dual variable $\mu$. This choice is motivated by our initial goal of minimizing the augmented Lagrangian function in both the primal and dual variables. However, throughout each proof, we accommodate the analysis to incorporate the penalty method update \eqref{eq: penalty-mu} whenever necessary. By adopting this approach, we ensure that our results are applicable to both the SDD and the penalty method, providing a comprehensive understanding of their convergence properties.}
\subsubsection{An $\mathcal{O}(\epsilon^{-4})$ Iteration Complexity} 
Due to that fact that $\nabla f$ is Lipschitz, we have $\|\nabla_i f(x_{<i}, x_i, x_{>i}) - \nabla_i f(x_{<i}, z_i, x_{>i})\| \leq L_f\|x_i-z_i\|$ for any $i\in [p]$, $x_i, z_i \in X_i$, and fixed $x_j$ for $j\neq i$;
%\begin{align*}
%	\|\nabla_i f(x_{<i}, x_i, x_{>i}) - \nabla_i f(x_{<i}, z_i, x_{>i})\| \leq L_f\|x_i-z_i\|.
%\end{align*}
this will be invoked several times in the analysis. We first show that the sequence $\{\mathcal{P}(x^k, \mu^k)\}_{k\in \N}$ is non-increasing.
\begin{lemma}[One-step Progress of SDD-ADMM]\label{lemma: admm-one-step-descent}
	Suppose Assumption \ref{assumption: mb problem data} holds. For all $k\in \N$, 
	\begin{align*}
		& 	\mathcal{P}(x^k, \mu^k) - \mathcal{P}(x^{k+1}, \mu^{k+1}) \\
		\geq & \left(\frac{\theta-1}{2}\right) \mathrm{Lip}(\mu^k, \rho) \sum_{i=1}^p \|x^{k+1}_i - x_i^{k}\|^2 +  \left(\tau + \frac{1}{2}\right)\frac{\omega}{\rho} \|\mu^{k+1}-\mu^k\|^2.
	\end{align*}
\end{lemma}
\begin{proof}
	We first establish the descent in $x$. Let $i\in [p]$, then it holds that
	\begin{align}
		& \mathcal{P}(x^{k+1}_{\leq i},x^k_{>i}, \mu^k) = \sum_{j \leq i} g_j(x^{k+1}_j)+  \sum_{j> i} g_j(x^{k}_j) + \mathcal{K}_{\rho}(x^{k+1}_{\leq i},x^k_{>i}, \mu^k) + \frac{\omega}{2\rho}\|\mu^k\|^2 \notag \\
		\leq  & \sum_{j \leq i} g_j(x^{k+1}_j)+  \sum_{j> i} g_j(x^{k}_j)  +  \mathcal{K}_{\rho}(x^{k+1}_{< i}, x^k_{\geq i}, \mu^k) + \langle \nabla_{x_i} \mathcal{K}_{\rho}(x^{k+1}_{< i}, x^k_{\geq i}, \mu^k), x^{k+1}_i-x_i^k\rangle \notag \\
		     &  + \frac{\mathrm{Lip}(\mu^k,\rho)}{2}\|x^{k+1}_i-x^k_i\|^2 +\frac{\omega}{2\rho} \|\mu^k\|^2 \notag  \\
		 \leq  & \sum_{j < i} g_j(x^{k+1}_j)+  \sum_{j\geq i} g_j(x^{k}_j) +   \mathcal{K}_{\rho}(x^{k+1}_{< i}, x^k_{\geq i}, \mu^k) +\frac{\omega}{2\rho}\|\mu^k\|^2 \notag  \\
		      & - \Big(\frac{\theta-1}{2}\Big)\mathrm{Lip}(\mu^k,\rho) \|x^{k+1}_i-x^k_i\|^2 \notag \\
		 = & \mathcal{P}(x^{k+1}_{< i},x^k_{\geq i}, \mu^k)- \left(\frac{\theta-1}{2}\right)\mathrm{Lip}(\mu^k,\rho) \|x^{k+1}_i-x^k_i\|^2, \notag 
	\end{align}
	where the first inequality is due to $\nabla_{x_i} \mathcal{K}_{\rho}(x, \mu^k)$ being Lipschitz and the second inequality is due to the optimality of $x_i^{k+1}$ in \eqref{eq: sdd-admm-x}. Summing the above inequality from $i=1$ to $p$, we have
	\begin{align}\label{eq: sdd-admm-descent-x}
			\mathcal{P}(x^k, \mu^k) - \mathcal{P}(x^{k+1}, \mu^{k}) \geq \left(\frac{\theta-1}{2}\right) \mathrm{Lip}(\mu^k, \rho) \sum_{i=1}^p \|x^{k+1}_i - x_i^{k}\|^2.	
	\end{align}
	Next we derive the descent in $\mu$. {The strong convexity of the objective in \eqref{eq: sdd-admm-mu} implies}
	\begin{align}\label{eq: sdd-admm-descent-mu}
		\mathcal{P}(x^{k+1}, \mu^{k}) - \mathcal{P}(x^{k+1}, \mu^{k+1}) \geq \left(\tau + \frac{1}{2}\right)\frac{\omega}{\rho}\|\mu^{k+1}-\mu^k\|^2.
	\end{align}
	{In view of \eqref{eq: penalty-mu}, the above inequality holds as well since $\mu^{k}=\mu^{k+1} = 0$.} Combining \eqref{eq: sdd-admm-descent-x} and \eqref{eq: sdd-admm-descent-mu} proves the lemma.
\end{proof}
\begin{remark}
	We assume that the proximal mapping \eqref{eq: prox-oracle-g}  of $g_i$ can be carried out exactly, which can be satisfied for many nonconvex functions such as SCAD, MCP, Capped-$\ell_1$, and indicator functions of sphere or annulus constraints. This assumption is mainly used to establish the descent property of $\mathcal{P}(x^{k+1}_{< i},\cdot, x^k_{>i}, \mu^k)$, whereas the global optimality of $x^{k+1}_i$ in \eqref{eq: sdd-admm-x} is not necessary. As an alternative, we may directly assume a descent oracle on $g_i$: we can find a stationary point $x_i^{k+1}$ of \eqref{eq: sdd-admm-x} such that 
	$\mathcal{P}(x^{k+1}_{< i},x^{k+1}_i ,x^k_{>i}, \mu^k)  \leq \mathcal{P}(x^{k+1}_{< i},x^{k}_i ,x^k_{>i}, \mu^k) -  \nu \|x^{k+1}_i-x^k_i\|^2 $
	for some $\nu > 0$. See also Assumption \ref{assumption: descent oracle}. This is in general more realistic when $g_i$ is highly complicated and reasonable if some nonconvex solver can be warm-started. 
\end{remark}

Next we show that the sequence $\{\mathcal{P}(x^k, \mu^k)\}_{k\in \N}$ is bounded from below; consequently, we can further control $\|h(x^k)\|$ and $\|\mu^k\|$. To this end, we require the infeasibility of the initial point $x^0$ to be controlled in the following sense. 
\begin{assumption}~\label{assumption: feasible initial pt}
	There exists a constant $C\geq 0$ such that for any finite $\alpha >0$, we can find an initial point $x^0\in X$ such that $\|h(x^0)\|^2 \leq \frac{C}{\alpha}$.
\end{assumption}
\begin{remark}\label{remark: feasibility}
Assumption \ref{assumption: feasible initial pt} is a slight relaxation of the requirement that $h(x^0) = 0$. As we will see in Theorem \ref{thm: SDD-ADMM}, in order to find an $\epsilon$-stationary solution, we will need to satisfy Assumption \ref{assumption: feasible initial pt} with  $\alpha = \rho = \Theta(\epsilon^{-2})$ and hence the initial point is required to be almost feasible, i.e., $\|h(x^0)\|=\mathcal{O}(\epsilon)$. It can be satisfied by solving a convex program when $X$ is convex and $h$ is affine, or when there exists $i\in [p]$ such that $h_i(x_i) = A_ix_i-b$ where $A_i$ has full row rank. We also note that this (near) feasibility assumption on the initial point is commonly adopted in the literature to establish iteration complexity estimates for nonlinear programs \cite{boob2022stochastic,lin2019inexact,xie2021complexity}. In addition, it suffices to replace $\overline{\mathcal{P}}$ defined in \eqref{eq: p_ub} by any finite upper bound on $f(x^0) + g(x^0)$ in case $f$ is not bounded from above over $X$.
\end{remark}

\begin{lemma}[Bounds on Dual Variable and Primal Residual] \label{lemma: admm-bound-primal-res}
	Suppose Assumptions \ref{assumption: mb problem data} and  \ref{assumption: feasible initial pt} hold. Recall  
	{$\overline{\mathcal{P}}$ and $\underline{\mathcal{P}}$ from \eqref{eq: upper and lower bound on P}, constant $C$ from Assumption \ref{assumption: feasible initial pt}, and further define
	\begin{align}\label{eq: delta p}
		 \Delta :=  \overline{\mathcal{P}} -\underline{\mathcal{P}} + \frac{C}{2},
	\end{align}
	which is a constant independent of the penalty $\rho$.}
	Then $\mathcal{P}(x^k, \mu^k) \geq  \underline{\mathcal{P}}$ for all $k\in \N$. Moreover, it holds that
	\begin{align}
		\|h(x^k)\| \leq  \left( \frac{ 4 \Delta }{\rho}	\right)^{1/2},  \ \text{and} \ \|\mu^k\| \leq & \left( \rho \Delta \right)^{1/2}. \label{eq: admm-h-bound}
	\end{align}
\end{lemma}
\begin{proof}
	Let $x^0$ be an initial point supplied to SDD-ADMM satisfying Assumption \ref{assumption: feasible initial pt} with $\alpha = \rho$. Moreover, since $\mu^0 = {0}$, we have 
	\begin{align*}
		\mathcal{P}(x^0, \mu^0) = & 
		f(x^0)+ \sum_{i=1}^p g_i(x^0_i) + \frac{\rho}{2}\|h(x^0)\|^2  \leq \overline{\mathcal{P}} + \frac{C}{2}.
	\end{align*}
	By Lemma \ref{lemma: admm-one-step-descent}, for all $k\in \N$, we have $\mathcal{P}(x^0, \mu^0)$ is greater than
	\begin{align}
		& \mathcal{P}(x^k, \mu^k) =  f(x^k)+ \sum_{i=1}^p g_i(x^k_i) + \langle \mu^{k}, h(x^k)\rangle + \frac{\rho}{2}\|h(x^k)\|^2 + \frac{\omega}{2\rho}\|\mu^k\|^2 \notag \\
		\geq & \inf_{x\in X} \biggl\{f(x) + \sum_{i=1}^p g_i(x_i)\biggr\} + \frac{\rho}{4}\|h(x^{k})\|^2 + \frac{1}{\rho}\|\mu^k\|^2 =  \underline{\mathcal{P}}  + \frac{\rho}{4}\|h(x^{k})\|^2 + \frac{1}{\rho}\|\mu^k\|^2 \geq \underline{\mathcal{P}}, \notag 
	\end{align}
	where the second inequality is due to $ \langle \mu^{k}, h(x^k)\rangle \geq -\frac{\rho}{4}\|h(x^k)\|^2 - \frac{1}{\rho} \|\mu^{k}\|^2$ and $\omega \geq 4$. The above inequality further gives the bounds in \eqref{eq: admm-h-bound}. 
\end{proof}
{Lemma \ref{lemma: admm-bound-primal-res} holds under both dual updates \eqref{eq: sdd-admm-mu} and \eqref{eq: penalty-mu}. If \eqref{eq: penalty-mu} is performed, then  \eqref{eq: admm-h-bound} can be improved to $\|h(x^k)\| \leq (2\Delta /\rho)^{1/2}$ and $\|\mu^k\| = 0$. }
Though $x^{k+1}_i$ is obtained by a single proximal gradient step, it still is an approximate stationary solution in the following sense.
\begin{lemma}[Bound on Dual Residual] \label{lemma: admm-bound-dual-res}
	Suppose Assumption \ref{assumption: mb problem data} holds. For all $k\in \N$ and $i\in [p]$, 
	\begin{align*}
		\mathrm{dist}\Big( -\nabla_i f(x^{k+1}) -\nabla h_i(x_i^{k+1}) \tilde{\mu}^{k+1}, \partial g_i(x_i^{k+1}) \Big) \leq & (\theta + 1 ) \mathrm{Lip}(\mu^k, \rho) \sum_{j\geq i} \|x_j^{k+1}-x_j^k\|,
	\end{align*}
	where $\tilde{\mu}^{k+1} := \mu^k+ \rho h(x^{k+1})$.
%	As a result, 
%	\begin{align*}
%		& \max_{i\in [p]} \Big\{ \mathrm{dist}\Big( \partial g_i(x_i^{k+1}), -\nabla_i f(x^{k+1}) -\nabla h_i(x^{k+1}) \tilde{\mu}^{k+1}\Big)\Big\} \\ \leq & (\theta + 1 ) \mathrm{Lip}(\mu^k, \rho) \sum_{i=1}^p \|x_i^{k+1}-x_i^k\|.
%	\end{align*}
\end{lemma}
\begin{proof}
	The update of $x^{k+1}_i$ gives
%	\begin{align*}
%		0 \in & \nabla_i f(x^{k+1}_{< i}, x_{\geq i}^k) + \partial g_i(x^{k+1}_i) +\nabla h_i(x_i^k)\left(\mu^k+ \rho \sum_{j< i} h_i(x_j^{k+1}) + \sum_{j\geq i} h_j(x_j^k) \right) \\
%		  & + \theta \mathrm{Lip}(\mu^k, \rho) (x^{k+1}_i - x_i^k).
%	\end{align*}
%	It follows that 
	$\xi^{k+1}_i \in \nabla_i f(x^{k+1}) +  \partial g_i(x^{k+1}_i) + \nabla h_i(x_i^{k+1}) \tilde{\mu}^{k+1}$, where 
	\begin{align*}
		 \xi^{k+1}_i & :=  \nabla_i f(x^{k+1}) - \nabla_i f(x^{k+1}_{< i}, x_{\geq i}^k) - \theta \mathrm{Lip}(\mu^k, \rho) (x^{k+1}_i - x_i^k) \notag \\
		&  + \nabla h_i(x_i^{k+1})(\mu^k + \rho h(x^{k+1})) - \nabla h_i(x_i^k)\Big(\mu^k+ \rho \sum_{j< i} h_i(x_j^{k+1}) + \rho \sum_{j\geq i} h_j(x_j^k) \Big). 
	\end{align*}
	The last two terms in the definition of $\xi_i^{k+1}$ can be bounded by 
	\begin{align*}
%		& \left\| \nabla h_i(x_i^{k+1})(\mu^k + \rho h(x^{k+1})) - \nabla h_i(x_i^k)\left(\mu^k+ \rho \sum_{j< i} h_i(x_j^{k+1}) + \rho \sum_{j\geq i} h_j(x_j^k) \right)\right\| \\
%		\leq 
		& \|(\nabla h_i(x^{k+1}_i)-\nabla h_i(x_i^k)) (\mu^k + \rho h(x^{k+1})\| +\rho  \| \nabla h_i(x_i^k) \|  \sum_{j\geq i} \|h_j(x_j^{k+1}) - h_j(x_j^k)\| \\
		\leq & (L_h \|\mu^k\| + \rho L_h M_h) \|x_i^{k+1}-x_i^k\|+ \rho J_h K_h \sum_{j\geq i} \|x_j^{k+1}-x_j^k\|;
	\end{align*}
	by the smoothness of $f$ and the definition of $\mathrm{Lip}(\mu^k, \rho)$ in \eqref{eq: admm-lip}, $\|\xi_i^k\|$ is bounded by
	\begin{align}
		& (L_f + 	\rho J_h K_h) \sum_{j\geq i} \|x_j^{k+1}-x_j^k\| + (L_h \|\mu^k\| + \rho L_h M_h + \theta \mathrm{Lip}(\mu^k, \rho))\|x_i^{k+1}-x_i^k\| \notag \\
		\leq & (\theta + 1 ) \mathrm{Lip}(\mu^k, \rho) \sum_{j\geq i} \|x_j^{k+1}-x_j^k\|. \notag 
	\end{align}
	This completes the proof. 
\end{proof}

With the help of the previous lemmas, we are now ready to present an iteration complexity upper bound for SDD-ADMM.
\begin{theorem}\label{thm: SDD-ADMM}
	Suppose Assumptions \ref{assumption: mb problem data} and \ref{assumption: feasible initial pt} hold, and let $\epsilon>0$. {Recall parameters $(M_h, K_h, J_h, L_h)$ from \eqref{eq: mb-h-constants} and $\Delta$ in \eqref{eq: delta p}, and define constants}
	\begin{align}
		\kappa_1 := J_hK_h + M_hL_h, \ {\kappa_2 :=  L_h \sqrt{\Delta } + \kappa_1 + 1.}
	\end{align}
	{Further choose $\rho \geq \max\{1, L_f, 4\Delta \epsilon^{-2}\}$,
	% \begin{align}
	% 	\rho \geq \max\{1, L_f, 4\Delta \epsilon^{-2}\}, 
	% \end{align}
	and let  $x^0\in X$ be an initial point satisfying Assumption \ref{assumption: feasible initial pt} with $\alpha = \rho$. }
Then SDD-ADMM {with input $(x^0,\rho, \omega, \theta, \tau)$} finds an $\epsilon$-stationary solution of \eqref{eq: mb-nlp} in at most $K(\rho)$ iterations, where
\begin{align}\label{eq: admm upper bd K}
		{K(\rho):=  \left\lceil \frac{2p \Delta (\theta+1)^2 \kappa^2_2  \rho}{(\theta-1) \kappa_1 \epsilon^2} \right\rceil }= \mathcal{O}(\rho \epsilon^{-2}).
	\end{align}
	In particular, if we choose $\rho= \Theta(\epsilon^{-2})$, then  {$K(\rho) = \mathcal{\mathcal{O}}(\epsilon^{-4})$.}
\end{theorem}
\begin{proof}
	We first show that $\mathrm{Lip}(\mu^k, \rho) = \Theta(\rho)$. {Since $\rho \geq \max\{1, L_f\}$, by the second inequality in \eqref{eq: admm-h-bound} of Lemma \ref{lemma: admm-bound-primal-res}, we have $\|\mu^k\| \leq \sqrt{\rho \Delta} \leq \rho \sqrt{\Delta}$} and 
	\begin{align}\label{eq: bound on lip}
		\rho\kappa_1   \leq \mathrm{Lip}(\mu^k, \rho) = L_f + \|\mu^k \| L_h + \rho (J_hK_h + M_hL_h) \leq {\rho\kappa_2. }
	\end{align}
	The above lower bound of $\mathrm{Lip}(\mu^k, \rho)$ and Lemma \ref{lemma: admm-one-step-descent} together give
	\begin{align*} 
		\frac{(\theta-1)\kappa_1}{2} \rho \sum_{i=1}^p \|x_i^{k+1}-x_i^k\|^2 \leq \mathcal{P}(x^k, \mu^k) - \mathcal{P}(x^{k+1}, \mu^{k+1}).
	\end{align*}
	Summing the above inequality from $k=0$ to some positive index $K-1$, we have
	\begin{align}\label{eq: bound sum x difference}
		\frac{(\theta-1)\kappa_1}{2} \rho \sum_{k=0}^{K-1} 	 \sum_{i=1}^p \|x_i^{k+1}-x_i^k\|^2  \leq \Delta.
	\end{align}
	As a result, there exists an index $0\leq \bar{k} \leq K-1$ such that 
	\begin{align}\label{eq: bound-on-sum-difference}
		\sum_{i=1}^p\|x_i^{\bar{k}+1}-x^{\bar{k}}_i\|\leq  \sqrt{p} \left( \sum_{i=1}^p \|x_i^{\bar{k}+1}-x^{\bar{k}}_i\|^2 \right)^{1/2} \leq \left( \frac{2p \Delta}{\rho (\theta-1)\kappa_1 K}\right)^{1/2}.
	\end{align}
	By \eqref{eq: admm-h-bound} in Lemma \ref{lemma: admm-bound-primal-res} and the choice that $\rho\geq 4\Delta\epsilon^{-2}$, we have $\|h(x^{\bar{k}+1})\| \leq \epsilon$. Moreover, recall $\tilde{\mu}^{\bar{k}+1} = \mu^{\bar{k}}+\rho h(x^{\bar{k}+1})$; Lemma \ref{lemma: admm-bound-dual-res}, the upper bound in \eqref{eq: bound on lip}, and \eqref{eq: bound-on-sum-difference} imply that
	\begin{align*}
		& \max_{i\in [p]} \Big\{ \mathrm{dist}\Big(  -\nabla_i f(x^{\bar{k}+1}) -\nabla h_i(x^{\bar{k}+1}) \tilde{\mu}^{\bar{k}+1}, \partial g_i(x_i^{\bar{k}+1})\Big) \Big\} \\ 
		\leq & (\theta + 1 ) \mathrm{Lip}(\mu^{\bar{k}}, \rho) \sum_{i=1}^p \|x_i^{\bar{k}+1}-x_i^{\bar{k}}\| \leq  (\theta + 1 ) {\kappa_2 \rho}  \left( \frac{2p \Delta}{\rho (\theta-1)\kappa_1 K}\right)^{1/2} \leq \epsilon,
	\end{align*}
	where the last inequality holds by the upper bound {$K=K(\rho)$} in \eqref{eq: admm upper bd K}. This completes the proof. 
\end{proof}

{In view of Lemma \ref{lemma: admm-bound-primal-res} and Theorem \ref{thm: SDD-ADMM}, the primal infeasibility is bounded by $\sqrt{4\Delta/\rho}$ while the dual infeasibility can be reduced to $\epsilon$ in $\mathcal{O}(\rho \epsilon^{-2})$ iterations. Such  measures can be informative if different primal and dual tolerances are preferred.} 

\subsubsection{Improve Iteration Complexity to $\mathcal{O}(\epsilon^{-3})$ and $\mathcal{O}(\epsilon^{-2})$}
{
Next we show that under an additional technical assumption, we can further improve the iteration complexity of SDD-ADMM. Given $r>0$ and $i\in [p]$, define 
\begin{align}
	X(r) := & \{x\in X~|~\|h(x)\| \leq r\}, \\
	X_i(r): = & \{x_i\in X_i~|~(x_{\neq i},x_i) \in X(r) \text{~for some $x_j\in X_j$, $j\in [p]\setminus\{i\}$}\}.
\end{align}
By Assumption \ref{assumption: feasible initial pt}, we know that $X(r)$ is nonempty for any $r>0$ and thus its projection $X_i(r)$ is also nonempty.
Now we further make the following assumption.
\begin{assumption}\label{assumption: improved rate}
	There exist $i\in [p]$, $(r, \sigma)\in \R^2_{++}$, and $(M_g, \nabla_f) \in \R_+^2$ such that 
	\begin{align}
		& \sigma \|\mu\| \leq \mathrm{dist}(-\nabla h_i(x_i) \mu, \partial g_i(x_i)) + M_g, \ \forall \mu \in \R^{m}, x_i \in X_i(r), \   \text{and} \ \label{eq: full rank} \\
		& \sup_{x\in X(r)} \|\nabla_i f(x)\| \leq \nabla_{f}. \label{eq: nabla f bounded}
	\end{align} 
\end{assumption}
\begin{remark}
	We give some comments regarding Assumption \ref{assumption: improved rate}. 
	\begin{enumerate}
		\item Suppose that $\nabla h_i(x_i)$ has full rank over $X_i(r)$, and their smallest singular values are bounded away from zero, i.e.,
		\begin{align}\label{eq: pos singular}
			\inf_{x_i\in X_i(r)} \sigma_{\min}(\nabla h_i(x_i)) > 0.
		\end{align}
        In Appendix \ref{examples}, we show that broad classes of  $g_i$ functions can ensure condition \eqref{eq: full rank} with the help of \eqref{eq: pos singular} or a similar constraint qualification. In particular, $g_i$ can be (Example \ref{example1}) a possibly nonconvex Lipschitz function, (Example \ref{example2}) a function of the form $\delta_{X_i} + \tilde{g}_i$ where $X_i$ is a sufficiently large full-dimensional closed convex set and $\tilde{g}_i$ is continuous and convex over $X_i$, or (Example \ref{example3}) an indicator function of a set defined by continuously differentiable constraints satisfying a constraint qualification. 
		\item Clearly, $h_i(x_i)=Ax-b$ with full row rank always implies \eqref{eq: pos singular} as $\sigma_{\min}(A)>0$. Even in this case, \eqref{eq: full rank} is still weaker than conditions (a) and (b) commonly adopted in existing ADMM works (reviewed in Section \ref{sec: admm lit}) as we allow the presence of some nonsmooth $g_i$.
		\item Condition \eqref{eq: nabla f bounded} is rather mild and can be satisfied under the boundedness of either $\nabla_i f$ or $X(r)$.
		\item A reasonable direction to further weaken Assumption \ref{assumption: improved rate} is to restrict the regions of $x_i$ and $x$ on which \eqref{eq: full rank} and \eqref{eq: nabla f bounded} hold. For example,  one can directly assume \eqref{eq: full rank} and \eqref{eq: nabla f bounded} hold on all algorithmic iterates $\{x^{k+1}\}_{k\in \N}$.
	\end{enumerate}
\end{remark}
We further comment on condition \eqref{eq: pos singular}. As a concrete example, consider $i=p=1$ and $h(x) = x^\top x - R$ for some $R>0$. Then given any $0 < r < R$, it holds that $ X(r) \subset \{x\in \R^n~|~ R-r \leq x^\top x \leq R + r\}$, and hence $$\sigma_{\min}(\nabla h(x)) = 2\|x\|\geq 2(R-r)^{1/2} > 0$$ for all $x\in X(r)$. In nonlinear programs, this condition is closely related to the well-known linearly independence constraint qualification (LICQ) commonly assumed on KKT points. 
% {Note that essentially our condition is imposed over $X_i(0)$, i.e., the feasible region of $x_i$, which can be justified by the Sard's Theorem, and then extended to $X_i(r)$ by the continuity of the rank of $\nabla h_i$. When nonlinear and nonconvex constraints, existing algorithms based on ALM/penalty methods \cite{li2021rate, Sahin2019alm, xie2021complexity, lin2019inexact}, sequential quadratic programs (SQP) \cite{berahas2021sequential, curtis2021inexact}, and proximal point methods (PPM) \cite{boob2022stochastic, ma2019proximally} all rely on certain regularity conditions to control dual variables. Condition \eqref{eq: pos singular} is different from all those used in the literature, generalizes the classic rank condition used for convergence of affine-constrained ADMM, and enables us to derive a new $\mathcal{O}(\epsilon^{-2})$ first-order iteration complexity estimate.} 
{
It is worth noting that our condition is primarily imposed on $X_i(0)$, the feasible region of $x_i$, which is justified by Sard's Theorem. Consequently, we extend this condition to $X_i(r)$ through the continuity of the rank of $\nabla h_i$. In the context of nonlinear and nonconvex constraints, existing algorithms such as those based on ALM/penalty methods \cite{li2021rate, Sahin2019alm, xie2021complexity, lin2019inexact}, sequential quadratic programs (SQP) \cite{berahas2021sequential, curtis2021inexact}, and proximal point methods (PPM) \cite{boob2022stochastic, ma2019proximally} all rely on specific regularity conditions to control the behavior of dual variables. In contrast, our condition \eqref{eq: pos singular} differs from those used in the literature. It not only generalizes the classic rank condition employed for the convergence of affine-constrained ADMM but also enables us to derive a novel first-order iteration complexity estimate of $\mathcal{O}(\epsilon^{-2})$. }

With Assumption \ref{assumption: improved rate}, we can derive new bounds on dual variables.
\begin{lemma}\label{lemma: new dual bounds}
	Suppose Assumptions \ref{assumption: mb problem data}, \ref{assumption: feasible initial pt}, and \ref{assumption: improved rate} hold. Further define constants 
	\begin{align}\label{eq: kappa_34}
		\kappa_3 := \frac{(\theta+1)\kappa_2}{\sigma} \sqrt{ \frac{2p\Delta}{(\theta-1)\kappa_1}}, 
%		\ \kappa_4 : = \frac{\kappa_3}{3} +  \frac{(1+\tau)\omega}{\omega-1} \sqrt{ \frac{2\Delta}{\omega} 
		\ \kappa_4 : = \frac{\kappa_3+ (1+\tau) \sqrt{2\Delta \omega}}{3}.
	\end{align}
	Suppose that $\rho \geq \max\{1, L_f, 4\Delta/r^2\}$. For any positive integer $K >0$, there exists an index $0\leq \bar{k} \leq K-1$ such that 
	\begin{align}\label{eq: new dual bounds}
		\| \tilde{\mu}^{\bar{k}+1}\|\leq \kappa_3 \sqrt{\frac{\rho}{K}} + \frac{\nabla_{f}+M_g}{\sigma}, \ \text{and} \ \|\mu^{\bar{k}}\| \leq \kappa_4 \sqrt{\frac{\rho}{K}} + \frac{\nabla_{f}+M_g}{3\sigma},
	\end{align}
	where $\tilde{\mu}^{\bar{k}+1} = \mu^{\bar{k}} +  \rho h(x^{\bar{k}+1})$.
\end{lemma}
\begin{proof}
	By Lemma \ref{lemma: admm-bound-primal-res}, the choice $\rho \geq 4\Delta/r^2$ ensures that $\{x^{k+1}\}_{k\in \N} \subset X(r)$, and hence Assumption \ref{assumption: improved rate} can be applied.
	Let $i\in [p]$ be the index specified in Assumption \ref{assumption: improved rate}. By Lemma \ref{lemma: admm-bound-dual-res}, there exists $\xi^{k+1}_i \in \R^{n_i}$ such that 
	\begin{align*}
		&\xi^{k+1}_i \in \nabla_i f(x^{k+1}) + \nabla h_i(x_i^{k+1}) \tilde{\mu}^{k+1} +  \partial g_i(x^{k+1}_i), \\
		& \|\xi^{k+1}_i\|\leq (\theta + 1 ) \mathrm{Lip}(\mu^k, \rho) \sum_{j=1}^p \|x_j^{k+1}-x_j^k\|.
	\end{align*}
	Hence, by Assumption \ref{assumption: improved rate} and the fact that $\mathrm{Lip}(\mu^k, \rho) \leq \kappa_2 \rho$ from \eqref{eq: bound on lip}, we have
	\begin{align}
		\sigma \|\tilde{\mu}^{k+1}\| \leq & \mathrm{dist}(-\nabla h_i(x_i) \tilde{\mu}^{k+1}, \partial g_i(x_i^{k+1})) + M_g\leq  \|\xi_i^{k+1}\| + \|\nabla f_i(x^{k+1})\| + M_g	\notag  \\
		\leq & (\theta + 1 ) \kappa_2 \rho \sum_{j=1}^p \|x_j^{k+1}-x_j^k\| + \nabla_{f}+M_g. \notag 
	\end{align}
	By \eqref{eq: bound-on-sum-difference} and the above inequality, we have 
	\begin{align}\label{eq: bound on true mu 2}
		\| \tilde{\mu}^{\bar{k}+1}\|
		\leq & \frac{(\theta+1)\kappa_2 \rho}{\sigma}  \left( \frac{2p \Delta}{\rho (\theta-1)\kappa_1 K}\right)^{1/2} + \frac{\nabla_{f}+M_g}{\sigma}.
	\end{align}
	{Hence the first inequality in \eqref{eq: new dual bounds} is proved for both dual updates \eqref{eq: sdd-admm-mu} and \eqref{eq: penalty-mu}. Since in the penalty method we have $\mu^k = 0$, it remains to prove the second inequality in \eqref{eq: new dual bounds} under the SDD update \eqref{eq: sdd-admm-mu}}. By the SDD update and the definition of $\tilde{\mu}^{k+1}$, we have
	\begin{align*}
		0 = & \rho h(x^{k+1}) + (1+\tau) \omega \mu^{k+1} - \tau \omega \mu^k = \tilde{\mu}^{k+1} + (1+\tau) \omega (\mu^{k+1}-\mu^k) + (\omega-1) \mu^k,
	\end{align*}
	which implies that 
	\begin{align}\label{eq: initial bound on mu_k}
		\|\mu^k\| \leq \frac{1}{\omega-1}\|\tilde{\mu}^{k+1}\| + \frac{(1+\tau)\omega}{\omega-1}\|\mu^{k+1}-\mu^k\|.
	\end{align}
	Next we bound the two terms on the right-hand side of \eqref{eq: initial bound on mu_k} at a specific index $\bar{k}$. By Lemma \ref{lemma: admm-one-step-descent} and a similar argument as in the proof of Theorem \ref{thm: SDD-ADMM}, we have for any positive integer $K$, it holds that 
	\begin{align*}
		\Delta 
  % \geq & \sum_{k=0}^{K-1} \left(\frac{(\theta-1)\kappa_1}{2} \rho \sum_{i=1}^p \|x^{k+1}_i - x_i^{k}\|^2 + \frac{\omega}{2\rho} \|\mu^{k+1}-\mu^k\|^2 \right) \\
		 \geq & K \left(\frac{(\theta-1)\kappa_1}{2} \rho \sum_{i=1}^p \|x^{\bar{k}+1}_i - x_i^{\bar{k}}\|^2 + \frac{\omega}{2\rho} \|\mu^{\bar{k}+1}-\mu^{\bar{k}}\|^2  \right),
	\end{align*}
	where
	\begin{align*}
		\bar{k } := \argmin_{k \in \{0, \cdots, K-1\}} \left\{ \frac{(\theta-1)\kappa_1}{2} \rho \sum_{i=1}^p \|x^{k+1}_i - x_i^{k}\|^2 + \frac{\omega}{2\rho} \|\mu^{k+1}-\mu^k\|^2 \right \},
	\end{align*}
	then we know that \eqref{eq: bound-on-sum-difference} holds and 
	\begin{align}\label{eq: bound dual differ}
		\|\mu^{\bar{k}+1} - \mu^{\bar{k}}\| \leq \left( \frac{2\rho \Delta}{\omega K}\right)^{1/2}.
	\end{align}
	Combining \eqref{eq: bound on true mu 2}, \eqref{eq: initial bound on mu_k}, and \eqref{eq: bound dual differ}, and , we have
	\begin{align*}
		\|\mu^{\bar{k}}\| \leq    \frac{(\theta+1)\kappa_2 \rho}{(\omega-1)\sigma}  \left( \frac{2p \Delta}{\rho (\theta-1)\kappa_1 K}\right)^{1/2} + \frac{(1+\tau)\omega}{\omega-1} \left( \frac{2\rho \Delta}{\omega K}\right)^{1/2} + \frac{\nabla_{f}+M_g}{(\omega-1)\sigma}.
	\end{align*}
	This completes the proof in view of $(\kappa_3, \kappa_4)$ defined in \eqref{eq: kappa_34} and the fact that $\omega\geq 4$.  
\end{proof}
\begin{theorem}\label{thm: SDD-ADMM improved}
	Suppose Assumptions \ref{assumption: mb problem data}, \ref{assumption: feasible initial pt}, and \ref{assumption: improved rate} hold, and let $\epsilon >0$. Recall $(r, \sigma, \nabla_{f}, M_g)$ is required in Assumption \ref{assumption: improved rate},   $(\kappa_3, \kappa_4)$ is defined in \eqref{eq: kappa_34}, and $K(\rho)$ is defined in \eqref{eq: admm upper bd K}. Choose $\rho \geq \max\{1, L_f, 4\Delta/r^2\}$.
	\begin{itemize}
		\item If $\nabla_{f}+M_g > 0$, then further let
		\begin{align}\label{eq: rho lb improved}
			\rho \geq  \left(\kappa_3+\kappa_4 + \frac{4(\nabla_{f}+M_g)}{3\sigma}\right) \epsilon^{-1},	
		\end{align}
		and let  $x^0\in X$ be an initial point satisfying Assumption \ref{assumption: feasible initial pt} with $\alpha = \rho$. Then SDD-ADMM with input $(x^0,\rho, \omega, \theta, \tau)$ finds an $\epsilon$-stationary solution of \eqref{eq: mb-nlp} in at most 
		\begin{align}
			K'(\rho): = \max\{ \lceil\rho \rceil, K(\rho) \}
		\end{align} iterations. In particular, if we choose $\rho = \Theta(\epsilon^{-1})$, then $K'(\rho) = \mathcal{O}(\epsilon^{-3})$. 
		\item If $\nabla_{f} = M_g = 0$, let  $x^0\in X$ be an initial point satisfying Assumption \ref{assumption: feasible initial pt} with $\alpha = \rho$. Then SDD-ADMM with input $(x^0,\rho, \omega, \theta, \tau)$ finds an $\epsilon$-stationary solution of \eqref{eq: mb-nlp} in at most 
		\begin{align}
			K''(\rho) := \max\{\lceil(\kappa_3+\kappa_4)^2 \epsilon^{-2}\rceil, K(\rho)\}
		\end{align}
		iterations. In particular, if we choose $\rho =\Theta(1)$, then $K''(\rho) = \mathcal{O}(\epsilon^{-2})$.
	\end{itemize}
\end{theorem}
\begin{proof}
	By a similar argument as in the proof of Theorem \ref{thm: SDD-ADMM}, at the index $\bar{k}$ specified in Lemma \ref{lemma: new dual bounds} with $K = K(\rho)$, the dual residual is bounded by $\epsilon$, i.e., 
	\begin{align*}
		\max_{i\in [p]} \Big\{ \mathrm{dist}\Big(  -\nabla_i f(x^{\bar{k}+1}) -\nabla h_i(x^{\bar{k}+1}) \tilde{\mu}^{\bar{k}+1}, \partial g_i(x_i^{\bar{k}+1})\Big) \Big\} \leq \epsilon.
	\end{align*}
	It remains to show $\|h(x^{\bar{k}+1})\| \leq \epsilon$. By the definition of $\tilde{\mu}^{\bar{k}+1}$, we have 
	\begin{align}\label{eq: bound primal improved}
		\|h(x^{\bar{k}+1})\| \leq &  \frac{1}{\rho}( \|\tilde{\mu}^{\bar{k}+1}\| + \|\mu^{\bar{k}}\|) \leq  \frac{1}{\rho} \left( (\kappa_3+\kappa_4)  \sqrt{\frac{\rho}{K}} + \frac{4(\nabla_{f}+M_g)}{3\sigma} \right),
%		\leq  & \frac{1}{\rho} \left( \kappa_3 + \kappa_4 + \frac{4\nabla_{f_i}}{3\sigma} \right) \leq \epsilon. \notag 
	\end{align}
	where the second inequality is due to Lemma \ref{lemma: new dual bounds}. Next we consider the two cases separately. 
	\begin{itemize}
		\item If $\nabla_{f}+M_g >0$, then \eqref{eq: bound primal improved} gives
		\begin{align}
			\|h(x^{\bar{k}+1})\| \leq  \frac{1}{\rho} \left( \kappa_3 + \kappa_4 + \frac{4(\nabla_{f}+M_g)}{3\sigma} \right) \leq \epsilon, \notag 
		\end{align}
		where the first inequality holds with any $K\geq \rho$, and the second inequality is due to the choice of $\rho$ in \eqref{eq: rho lb improved}.
		\item  If $\nabla_{f}=M_g = 0$, then  \eqref{eq: bound primal improved}  gives
		\begin{align}
			\|h(x^{\bar{k}+1})\| \leq  (\kappa_3+\kappa_4)  \sqrt{\frac{1}{\rho K}} \leq  (\kappa_3+\kappa_4) \sqrt{\frac{1}{ K}}  \leq \epsilon,  \notag
		\end{align}
		which the second inequality is due to $\rho\geq 1$ and the last inequality holds with any $K \geq (\kappa_3+\kappa_4)^2 \epsilon^{-2}$.
	\end{itemize}
	This completes the proof.
\end{proof}
\begin{remark}
	In view of Examples \ref{example2} and \ref{example3}, it is possible to have $M_g=0$ with $g_i$ being the indicator function of some proper sets. While the condition $\nabla_f=0$ is a little restrictive as it means that $f$ is constant with respect to $x_i$, this is not impossible as we work with multi-block problems. As a result, our $\mathcal{O}(\epsilon^{-2})$ complexity estimate in Theorem \ref{thm: SDD-ADMM improved}, if not stronger than, complements the previous ADMM works in the sense that we do not reply on both conditions (a) and (b) discussed in Section \ref{sec: admm lit}.
\end{remark}
}

\subsection{{Adaptive SDD-ADMM} \label{sec: adaptive admm}}
{In Algorithm \ref{alg: SDD-ADMM},  we use a fixed penalty $\rho$, which is in the order of $\Theta(\epsilon^{-2})$ in view of Theorem \ref{thm: SDD-ADMM}, or $\Theta(\epsilon^{-1})$ and $\Theta(1)$ in view of Theorem \ref{thm: SDD-ADMM improved}. The exact value of $\rho$ depends on the problem data, i.e., parameters $(L_f, M_h, K_h, J_h, L_h)$ and $(r, \sigma, \nabla_f, M_g)$ required in Assumption \ref{assumption: improved rate}, and may not be straightforward to estimate for some applications. In this subsection, we show that it is possible to find an $\epsilon$-stationary point of problem \eqref{eq: mb-nlp} through multiple calls of SDD-ADMM with increasing $\rho$'s. Moreover, this adaptive version does not deteriorate the iteration estimates established in Theorems \ref{thm: SDD-ADMM} and \ref{thm: SDD-ADMM improved}.}	
\begin{algorithm}[H]
	\caption{: \texttt{Adaptive SDD-ADMM}} \label{alg: adaptive SDD-ADMM}
	\begin{algorithmic}[1]
		 \State \textbf{Input} $(\rho_0, \omega, \theta, \tau, \epsilon) \in (0, +\infty) \times [4, +\infty) \times (1, \infty) \times[0, +\infty) \times (0, +\infty)$; 
		 \State \textbf{initialize} index $t \gets 0$;
		\While{an $\epsilon$-stationary solution of \eqref{eq: mb-nlp} is not found}
		\State $t\gets t+1$;
		\State find $x^{t,0}\in X$ satisfying Assumption \ref{assumption: feasible initial pt} with $\alpha = \rho_t:  = 2^t \rho_0$;
		\State run \texttt{SDD-ADMM}$(x^{t,0}, \rho_t, \omega, \theta, \tau)$ for at most $K(\rho_t)$ iterations;
		\EndWhile
	\end{algorithmic}
\end{algorithm}
{The proposed adaptive version of SDD-ADMM is presented in Algorithm \ref{alg: adaptive SDD-ADMM}. Essentially we start SDD-ADMM with a relatively small penalty $\rho_t$ for some iterations, and rerun SDD-ADMM with $\rho_{t+1} = 2\rho_t$ until an $\epsilon$-stationary solution is located. One technical issue is that, we need to initialize the $t$-th SDD-ADMM with a proper $x^0\in X$ satisfying Assumption \ref{assumption: feasible initial pt} with $\alpha = \rho_t$. Of course, if some $x^0 \in X\cap \{x|h(x^0)=0\}$ is available, then we can set $x^{t,0} = x^0$ for all index $t\geq 1$. Otherwise, any primal iterate in the $t$-th SDD-ADMM satisfying  Assumption \ref{assumption: feasible initial pt} with $\alpha = \rho_{t+1}$ can serve as $x^{t+1,0}$.}

{Though invoking a sequence of calls to SDD-ADMM, this adaptive version preserves the same iteration complexities established in Theorems \ref{thm: SDD-ADMM} and \ref{thm: SDD-ADMM improved}. To see this, denote the total number of SDD-ADMM calls by $T$. Notice that in view of Theorem \ref{thm: SDD-ADMM}, there exists a constant $B>0$ such that $K(\rho) \leq B \rho \epsilon^{-2}$. The total number of SDD-ADMM iterations can be then bounded by 
\begin{align}\label{eq: total sdd-admm}
	\mathcal{T}: = \sum_{t=1}^T K(\rho_t) \leq B\epsilon^{-2} \times  \sum_{t=1}^T \rho_0 2^t = 	B\epsilon^{-2} \times 2\rho_0(2^T-1) \leq  2B\epsilon^{-2}\rho_02^T.
\end{align}
	By Theorem \ref{thm: SDD-ADMM}, it suffices to find $T$ such that $\rho_T = \rho_0 2^T = \Theta(\epsilon^{-2})$, plugging which into \eqref{eq: total sdd-admm} gives the same $\mathcal{T} = \mathcal{O}(\epsilon^{-4})$ iteration complexity estimate. Similarly under assumptions of Theorem \ref{thm: SDD-ADMM improved}, the orders of $K'(\rho) =\mathcal{O}(\epsilon^{-3})$ and $K''(\rho) =\mathcal{O}(\epsilon^{-2})$ can be preserved as well.
}

\section{Unscaled Dual Descent ALM}\label{sec: udd}
The success of SDD motivates us to ask a natural question: what if we skip the scaling step? In this section, we investigate the unscaled dual descent update for solving the following special case of problem \eqref{eq: mb-nlp}, where $p=1$, $g$ is convex, and $h$ is affine:
\begin{align}\label{eq: lc problem}
	\min_{x\in \R^n} \left\{ f(x) + g(x) ~|~ h(x):=Ax-b = 0 \right\},
\end{align}
We note that the analysis in this section can be applied to a more general multi-block setting, while focusing on $p=1$ suffices to demonstrate the behavior of unscaled dual descent. Different from SDD-ADMM, the convergence of UDD-ALM requires certain regularity or constraint qualification to hold at the primal limit point, so that the sequence of dual variables has a bounded subsequence and the augmented Lagrangian function may then serve as a potential function. {We focus on the structured setup \eqref{eq: lc problem} in this section. In Appendix \ref{sec: udd-nonlinear}, we generalize the analysis to handle a more challenging setting with nonconvex $g$ and nonlinear $h$ by assuming a stronger subproblem oracle.}

Formally, we adopt the following assumptions.
\begin{assumption}\label{assumption: udd}
	We make the following assumptions regarding problem \eqref{eq: lc problem}.
	\begin{enumerate}
		\item The function $g:\R^{n} \rightarrow \overline{\R}$ can be decomposed as $g_0 + \delta_{X}$, where $X \subseteq \R^n$ is convex and compact, and $g_0:\R^{n} \rightarrow \R$ is continuous and convex over $X$. 
		\item The function $f:\R^{n} \rightarrow \R$ has $L_f$-Lipschitz gradient over $X$.
		\item The constraints $h(x) = Ax-b$ are affine with $A\in \R^{m\times n}$ and $b\in \R^m$, and $X\cap\{x|Ax=b\} \neq \emptyset$. 
	\end{enumerate}
\end{assumption}

Recall the definition of $\mathcal{K}_{\rho}$ in \eqref{eq: smooth AL}. We see $\nabla_x \mathcal{K}_{\rho}(x, \mu) = \nabla f(x) + A^\top \mu + \rho A^\top (Ax-b)$ is Lipschitz with modulus $L_f + \rho \|A^\top A\|$,
%\begin{align}
%	L_{\mathcal{K}}:= 	L_f + \rho \|A^\top A\|,
%\end{align}
which is independent of $\mu$ due to the linearity of constraints. This fact allows us to use a single proximal gradient step to update $x$. See Algorithm \ref{alg: UDD-ALM}. 
\begin{algorithm}[h!]
	\caption{: \texttt{UDD-ALM} for problem \eqref{eq: lc problem}} \label{alg: UDD-ALM}
	\begin{algorithmic}[1]
        \State \textbf{Initialize} $x^0\in X$, $\mu^0 \in \R^m$, $\rho\geq 0$, $\varrho > 0$, and $\theta >1$; set $L_{\mathcal{K}} = L_f + \rho \|A^\top A\|$;
		\For{$k=0,1,2\cdots$}
		\State perform a proximal gradient step:
		\begin{align}\label{eq: udd-alm-x}
			x^{k+1} 
%			= & \argmin_{x\in \R^n} g(x) + \langle \nabla_x \mathcal{K}_{\rho}(x^k, \mu^k), x-x^k \rangle + \frac{\theta L_\mathcal{K}}{2}\|x-x^k\|^2 \notag \\
			= &\argmin_{x\in \R^n}  g(x) + \langle \nabla f(x^k)+ A^\top (\mu^k + \rho (Ax^{k}-b)), x-x^k \rangle + \frac{\theta L_\mathcal{K}}{2}\|x-x^k\|^2;
		\end{align}
		\State update $\mu^{k+1}$ through a unscaled dual descent update: 
		\begin{align}\label{eq: udd-alm-dual}
			\mu^{k+1} = \mu^k -  \varrho (Ax^{k+1}-b);
		\end{align}
		\EndFor
	\end{algorithmic}
\end{algorithm}

%\subsubsection{Convergence of UDD-ALM with Affine Constraints}
Firstly we establish the descent of the augmented Lagrangian function.
\begin{lemma}[One-step Progress of UDD-ALM]\label{lemma: udd one step descent}
	Suppose Assumption \ref{assumption: udd} holds. For all $k\in \Z_{++}$, we have
	\begin{align} %\label{eq: udd one step descent}
	 	L_{\rho}(x^{k}, \mu^k) -	L_{\rho}(x^{k+1},\mu^{k+1})  \geq  \left(\frac{2\theta-1}{2}\right)  L_{\cal{K}}  \|x^{k+1}-x^k\|^2 + \varrho\|Ax^{k+1}-b\|^2.	
	\end{align}
\end{lemma}
\begin{proof}
	Similar to \eqref{eq: sdd-admm-descent-x} and using the fact that $g$ is convex, the descent in $x$ is given as $ L_{\rho}(x^{k}, \mu^k) - L_{\rho}(x^{k+1}, \mu^k)  \geq \left(\frac{2\theta-1}{2}\right) L_{\cal{K}}\|x^{k+1}-x^k\|^2.$
%	\begin{align*}
%		 L_{\rho}(x^{k}, \mu^k) - L_{\rho}(x^{k+1}, \mu^k)  \geq \left(\frac{2\theta-1}{2}\right) L_{\cal{K}}\|x^{k+1}-x^k\|^2.
%	\end{align*}
	The change with respect to $\mu$ is given as
% 	\begin{align*}
%		  L_{\rho}(x^{k+1},\mu^{k+1})-L_{\rho}(x^{k+1},\mu^{k}) = \langle \mu^{k+1}-\mu^k, Ax^{k+1}-b) \rangle = -\varrho\|Ax^{k+1}-b\|^2,
%	\end{align*}
	 $L_{\rho}(x^{k+1},\mu^{k+1})-L_{\rho}(x^{k+1},\mu^{k}) = \langle \mu^{k+1}-\mu^k, Ax^{k+1}-b) \rangle = -\varrho\|Ax^{k+1}-b\|^2$,
	where the last equality is due to the unscaled dual descent update. Combining the inequality and the equality proves the claim.
\end{proof}

We then bound the dual residuals of iterates produced by UDD-ALM.
\begin{lemma}[Bound on Dual Residual in UDD-ALM] \label{lemma: udd bound on dual res}
	Suppose Assumption \ref{assumption: udd} holds. For all $k\in \N$, it holds that 
\begin{align*}
	 & \mathrm{dist}\left( \partial g(x^{k+1}), -\nabla f(x^{k+1})- A^\top \mu^{k+1} \right) \\
	 \leq & (\theta+1) L_{\cal{K}} \|x^{k+1}-x^k\|  + (\rho+\varrho) \|A\|\|A x^{k+1}-b\|.
\end{align*}	
\end{lemma}
\begin{proof}
    {The claim follows from the optimality of of $x^{k+1}$ in \eqref{eq: udd-alm-x}, the fact that $\nabla_x \mathcal{K}_{\rho}$ is Lipschitz, and straightforward derivations.}
% 	The update of $x^{k+1}$ in \eqref{eq: udd-alm-x} gives
% %	\begin{align*}
% %		0 \in  \nabla f(x^k) + \partial g(x^{k+1}) + A^\top (\mu^k + \rho (Ax^k-b)) + \theta L_\mathcal{K} (x^{k+1} -x^k).
% %	\end{align*}
% %	It follows that
% 	$\xi^{k+1} \in 	\nabla f(x^{k+1}) + \partial g(x^{k+1}) + A^\top \mu^{k+1}$, where
% 	\begin{align*}
% 		\xi^{k+1} := & (\nabla f(x^{k+1}) - \nabla f(x^{k})) - \theta L_\mathcal{K}(x^{k+1}-x^k) -\varrho A^\top (Ax^{k+1}-b)-\rho A^\top (Ax^{k}-b)\\
% 		= &  (\nabla f(x^{k+1}) - \nabla f(x^{k})) - \theta L_\mathcal{K} (x^{k+1}-x^k) -(\rho+\varrho) A^\top (Ax^{k+1}-b) \\
% 		  & + \rho A^\top A(x^{k+1}-x^k)
% 	\end{align*}
% 	As a result, we have $\|\xi^{k+1}\| \leq (\theta+1) L_{\cal{K}} \|x^{k+1}-x^k\|  + (\rho+\varrho)\|A\|\|A x^{k+1}-b\|.$
% 	% \begin{align*}
% 	% 	\|\xi^{k+1}\| \leq (\theta+1) L_{\cal{K}} \|x^{k+1}-x^k\|  + (\rho+\varrho)\|A\|\|A x^{k+1}-b\|.
% 	% \end{align*}
% 	This completes the proof.
\end{proof}

Lemma \ref{lemma: udd one step descent} suggests that values of the augmented Lagrangian function form a non-increasing sequence. We aim to show that this sequence is actually bounded from below if certain regularity condition is satisfied.
\begin{definition}[Modified Robinson's Condition]\label{def: robinson}
	We say $x\in X=\mathrm{dom}~g$ satisfies the modified Robinson's condition if $\{ Ad~|~d \in T_{X}(x) \} = \R^m$,
%	\begin{align*}
%		\{ Ad~|~d \in T_{X}(x) \} = \R^m,
%	\end{align*}
	where $T_{X}(x)$ denotes the tangent cone of $X$ at $x$:
	\begin{align*}
		T_{X}(x) = \left\{d\in \R^n~|~d = \lim_{k\rightarrow\infty}\frac{x^k-x}{\tau_k}, x^k\rightarrow x, \tau_k\downarrow 0, \{x_k\}_{k\in \N}\subseteq X \right\}.
	\end{align*}
\end{definition}
The above definition is slightly different from the standard Robinson's condition, e.g., in \cite[Section 3.3.2]{ruszczynski2011nonlinear}: we do not require $x$ to satisfy $Ax=b$ in Definition \ref{def: robinson}. 
Despite this difference, \cite[Lemma 3.16]{ruszczynski2011nonlinear} still gives a sufficient condition: $A$ has full row rank and $x + \mathrm{Null}(A) \cap \mathrm{int}~X \neq \emptyset$, where $\mathrm{Null}(A)$ denotes the null space of $A$ and $\mathrm{int}~X $ denotes the interior of $X$. This modified Robinson's condition has been adopted to ensure certain boundedness condition on the dual sequence and verified in specific applications \cite{hajinezhad2019perturbed,shi2020penalty1,shi2020penalty2}. Since the techniques are not new, we present the following lemma in the context of problem \eqref{eq: lc problem} {while skipping the proof.}
\begin{lemma}[Existence of Dual Limit Point]\label{lemma: exist of dual limit point}
	Suppose Assumption \ref{assumption: udd} holds. Let $x^*\in X$ be a limit point of the sequence $\{x_k\}_{k\in \N}$ generated by  UDD-ALM, and $\{x^{k_r}\}_{r\in \N}$ be the subsequence convergent to $x^*$. If $x^*$ satisfies the modified Robinson's condition, then $\{\mu^{k_r} \}_{r\in \R^n}$ has a bounded subsequence and hence a limit point $\mu^*$.
\end{lemma}

\begin{theorem}\label{thm: udd convergence}
	Suppose Assumption \ref{assumption: udd} holds. Let $x^*$ be a limit point of the sequence $\{x^{k}\}_{k\in \N}$ generated by UDD-ALM that satisfies the modified Robinson's condition. Then the following statements hold.
	\begin{enumerate}
		\item (Asymptotic Convergence) The point $x^*$ is a stationary point of problem \eqref{eq: lc problem}.
		\item (Iteration Complexity) Let $\epsilon>0$. Define constants
		\begin{align*}
			\delta_1 := \min \left \{\frac{(2\theta-1)L_{\cal{K}}}{2}, \varrho \right\}, \quad \delta_2 := (\theta+1) L_{\cal{K}} + (\rho+\varrho)\|A\|. 
		\end{align*}
		UDD-ALM finds an $\epsilon$-stationary solution in at most $K$ iterations where 
		\begin{align}\label{eq: udd iter complex}
			K \leq \left\lceil \frac{\max\{1, \delta_2\}^2(L_{\rho}(x^0, \mu^0) - f(x^*)-g(x^*))}{\delta_1 \epsilon^2} \right\rceil = \mathcal{O}(\epsilon^{-2}).
		\end{align}
	\end{enumerate}
\end{theorem}
\begin{proof}
	Let $\{x^{k_r}\}_{r\in \N}$ be the subsequence convergent to $x^*$. By Lemma \ref{lemma: exist of dual limit point}, we may assume $\mu^{k_r}\rightarrow \mu^* \in \R^m$ as $r\rightarrow \infty$ without loss of generality. Consequently, $L_{\rho}(x^{k_r}, \mu^{k_r})\rightarrow L_{\rho}(x^*, \mu^*)$ due to the continuity of the augmented Lagrangian function over $X$. By Lemma \ref{lemma: udd one step descent}, the sequence $\{L_{\rho}(x^k, \mu^k)\}_{r\in \N}$ is non-increasing, so the whole sequence is bounded from below by $L_{\rho}(x^*, \mu^*)$. Summing the inequality claimed in  Lemma \ref{lemma: udd one step descent} from $k=0$ to some positive integer $K-1$, we have
	\begin{align}\label{eq: telescope}
		& \min \left \{\frac{(2\theta-1)L_{\cal{K}}}{2}, \varrho \right\} 	 \sum_{k=0}^{K-1} \|x^{k+1}-x^k\|^2 + \|Ax^{k+1}-b\|^2 \notag \\
		\leq & \sum_{k=0}^{K-1} \frac{(2\theta-1)L_{\cal{K}}}{2}\|x^{k+1}-x^k\|^2 + \varrho \|Ax^{k+1}-b\|^2 \leq  L_{\rho}(x^0,\mu^0) -L_{\rho}(x^*,\mu^*). 
	\end{align}
	Now we are ready to prove the two claims. 
	\begin{enumerate}
		\item Let $K\rightarrow \infty$ in \eqref{eq: telescope} and focusing on the subsequence $\{x^{k_r}, \mu^{k_r}\}_{r\in\N}$, we see $\lim_{r\rightarrow \infty} \max\{ \|x^{k_r}-x^{k_r-1}\|, \|A x^{k_r}-b\| \} = 0$,
%		\begin{align*}
%			\lim_{r\rightarrow \infty} \max\{ \|x^{k_r}-x^{k_r-1}\|, \|A x^{k_r}-b\| \} = 0,
%		\end{align*}
		which immediately follows that $\|Ax^*-b\| = \lim_{r\rightarrow \infty}\|Ax^{k_r}-b\|  = 0.$
%		\begin{align*}
%			\|Ax^*-b\| = \lim_{r\rightarrow \infty}\|Ax^{k_r}-b\|  = 0.
%		\end{align*}
		Moreover, by Lemma \ref{lemma: udd bound on dual res}, 
		\begin{align*}
			& \mathrm{dist}\left(  -\nabla f(x^{*})- A^\top \mu^{*}, \partial g(x^{*}) \right)\leq  \lim_{r\rightarrow \infty} \mathrm{dist}\left( -\nabla f(x^{k_r})- A^\top \mu^{k_r},  \partial g(x^{k_r}) \right) \\
			\leq & \lim_{r\rightarrow \infty}    (\theta+1) L_{\cal{K}} \|x^{k_r}-x^{k_r-1}\|  + (\rho+\varrho) \|A\|\|A x^{k_r}-b\| = 0.
		\end{align*}
		This suggests that $x^*$ is a stationary point of problem \eqref{eq: lc problem}.
		\item Since $Ax^*-b=0$, we have $L_{\rho}(x^*, \mu^*) = f(x^*)+ g(x^*)$. By \eqref{eq: telescope}, there exists an index $0\leq \bar{k} \leq K-1$ such that 
		\begin{align*}
			\max\{ \|x^{\bar{k}+1}-x^{\bar{k}}\|, \|Ax^{\bar{k}+1}-b\|\}	 \leq \left(\frac{L_{\rho}(x^0, \mu^0) - f(x^*)-g(x^*)}{\delta_1 K}\right)^{1/2}.
		\end{align*}
		By Lemma \ref{lemma: udd bound on dual res} and the above inequality,
		\begin{align*}
			& \max\left\{\mathrm{dist}\left( \partial g(x^{\bar{k}+1}), -\nabla f(x^{\bar{k}+1})- A^\top \mu^{\bar{k}+1} \right), \|Ax^{\bar{k}+1}-b\|\right\}	 \\
			% \leq &\max \left \{ (\theta+1) L_{\cal{K}}  \|x^{\bar{k}+1}-x^{\bar{k}}\|  + (\rho+\varrho) \|A\|\|A x^{\bar{k}+1}-b\|, \|Ax^{\bar{k}+1}-b\|  \right\}\\
			% \leq & \max\{1,\delta_2\}  \max\{ \|x^{\bar{k}+1}-x^{\bar{k}}\|, \|Ax^{\bar{k}+1}-b\|\} \\
			\leq & \max\{1,\delta_2\}  \left(\frac{L_{\rho}(x^0, \mu^0) - f(x^*)-g(x^*)}{\delta_1 K}\right)^{1/2} \leq \epsilon,
		\end{align*}	
		where the last inequality holds by the claimed upper bound of $K$ in \eqref{eq: udd iter complex}. 
	\end{enumerate}
\end{proof}
\begin{remark}
	We give some remarks on UDD-ALM. 
	\begin{enumerate}
		\item Different from the $\mathcal{O}(\epsilon^{-2})$ established in \cite{zhang2020global, zhang2020proximal}, Theorem \ref{thm: udd convergence} relies on the modified Robinson's condition at the limit point of iterates produced by UDD-ALM. Though assuming a certain constraint qualification at the limit point is common in nonlinear programs, this specific condition may not be satisfied by general instances of \eqref{eq: lc problem}.  
		\item In Appendix \ref{sec: udd-nonlinear}, we extend the iteration complexity result to handle nonconvex $g$ and nonlinear $h$ by assuming a stronger subproblem oracle. 
	\end{enumerate}
\end{remark}

The value of  $\varrho$ deserves more attentions in deriving the $\mathcal{O}(\epsilon^{-2})$ complexity in Theorem \ref{thm: udd convergence}. We treat $\varrho$ as a constant in our analysis and do not impose explicit requirements. However, it is observed that a larger $\varrho$ usually leads to iterates staying on the boundary of $X$, and hence the limit point is more likely to violate the modified Robinson's condition. As we illustrate in the Section \ref{section: udd example}, the behavior of UDD-ALM is very sensitive to the choice of $\varrho$. For certain instances, the numerical value of $\varrho$ needs to be even smaller than $\epsilon$ in order for UDD-ALM to exhibit convergence. In this case, the $\mathcal{O}(\epsilon^{-2})$ complexity may not be practically informative. We share more empirical observations in Section \ref{section: udd example}.

\section{Numerical Experiments}\label{section: examples}
\subsection{{SDD-ALM for Nonconvex QCQP}}
{
In this section, we consider the following nonconvex quadratically constrained quadratic program
\begin{align}\label{eq: qcqp}
	\min_{x\in \R^n}~\{f(x) := x^\top Q x + q^\top x~|~h(x):= x^\top Bx- 1 = 0, \|x\| \leq r\}, 
\end{align}
and compare with the iALM proposed in \cite{li2021rate}. Let $g(x) = \delta_{\{x|\|x\|\leq r\}}(x)$, whose projection operator can be computed explicitly. We generate data as follows: first create $\tilde{Q} \in \R^{n\times n}$ with standard Gaussian entries, and set $Q = 0.5 (\tilde{Q} + \tilde{Q}^\top)$; generate $\bar{B}$ in the same was as $Q$, and set $B = \bar{B} + (\|\bar{B}\|+1)I_{n}$, where $I_n$ denotes the identity matrix; finally we set $q$ to be the zero vector and $r = n/10$. For SDD-ADMM, we choose $(\omega, \theta, \tau) = (4,2,1)$ and simply set $\rho = 10n$. For iALM, we limit the number of outer-level updates by 10, and the penalty used in each outer level is $\beta_k = \beta_0 \sigma^k$, where $\beta_0 = 1$ and $\sigma = 2(\rho/\beta_0)^{1/10}$, so that $\beta_k$ in iALM should be able to quickly catch up the SDD-ALM penalty $\rho$; we set the input tolerance to the inner-level APG and middle-level iPPM to be 1e-3.
% and restrict the total number of APG iterations in each iPPM by 10,000; 
% all other implementation details follow the description in \cite{li2021rate}. 
For both algorithms, we first generate a vector with standard Gaussian entries, then scale it to get $x^0$ so that $\|h(x^0)\|= 0.5/\sqrt{\rho}$. }

{For each run of an algorithm, we record the primal residual ``\texttt{pres}" (measured by $\|h(x^{k+1})\|$), dual residual ``\texttt{dres}" (measured by $\|x^{k+1}-x^k\|$), the iteration index ``\texttt{iter}" when both \texttt{pres} and \texttt{dres} drop below 1e-3 for the first time, as well as the wall clock time ``\texttt{time}" over 100,000 proximal gradient iterations; if either \texttt{pres} or \texttt{dres} does not drop below 1e-3, we record their values where the sum of \texttt{pres} and \texttt{dres} is the minimum, and set \texttt{iter}=100,000. For $n\in \{100, 200, 300\}$, we generate 5 instances and report the average metrics in Table \ref{table1}. For $n\in \{100, 200\}$, SDD-ALM reduces \texttt{pres} below 1e-3, while for $n=300$, SDD-ALM achieves a slightly better \texttt{pres}. In constrast, iALM maintains a smaller \texttt{dres} in all runs. On average, SDD-ALM takes less time to perform 100,000 proximal gradient updates. }
\begin{table}[h!]
\caption{Averaged Metrics of SDD-ALM and iALM \cite{li2021rate}}\label{table1}
\begin{center}
\begin{tabular}{lllllllll}
\toprule
    & \multicolumn{4}{l}{SDD-ALM}        & \multicolumn{4}{l}{iALM}             \\
 \hline 
$n$   & \texttt{pres} & \texttt{dres} & \texttt{iter} & \texttt{time} & \texttt{pres} & \texttt{dres} & \texttt{iter} &\texttt{time} \\
100 & 1.00e-3    & 2.19e-8 & 16,158  & 13.97 & 1.03e-2 & 2.68e-10 & 100,000 & 25.73  \\
200 & 1.00e-3    & 3.02e-8 & 81,729  & 37.08 & 1.03e-2 & 6.51e-13 & 100,000 & 73.37  \\
300 & 3.11e-3    & 2.97e-9 & 100,000 & 69.43 & 8.86e-3 & 6.78e-13 & 100,000 & 147.29 \\
\bottomrule
\end{tabular}
\end{center}
\end{table}

{For the generated instances, we also observe that iALM usually reduces the objective value faster than SDD-ALM, while SDD-ALM seems to converge to solutions with better qualities in the long run. The objective trajectories of two instances are plotted in Fig. \ref{figure: qcqp_obj}.}
\begin{figure*}[h!]
\center{
\begin{tabular}{@{}c@{}}
	\includegraphics[width=.42\linewidth]{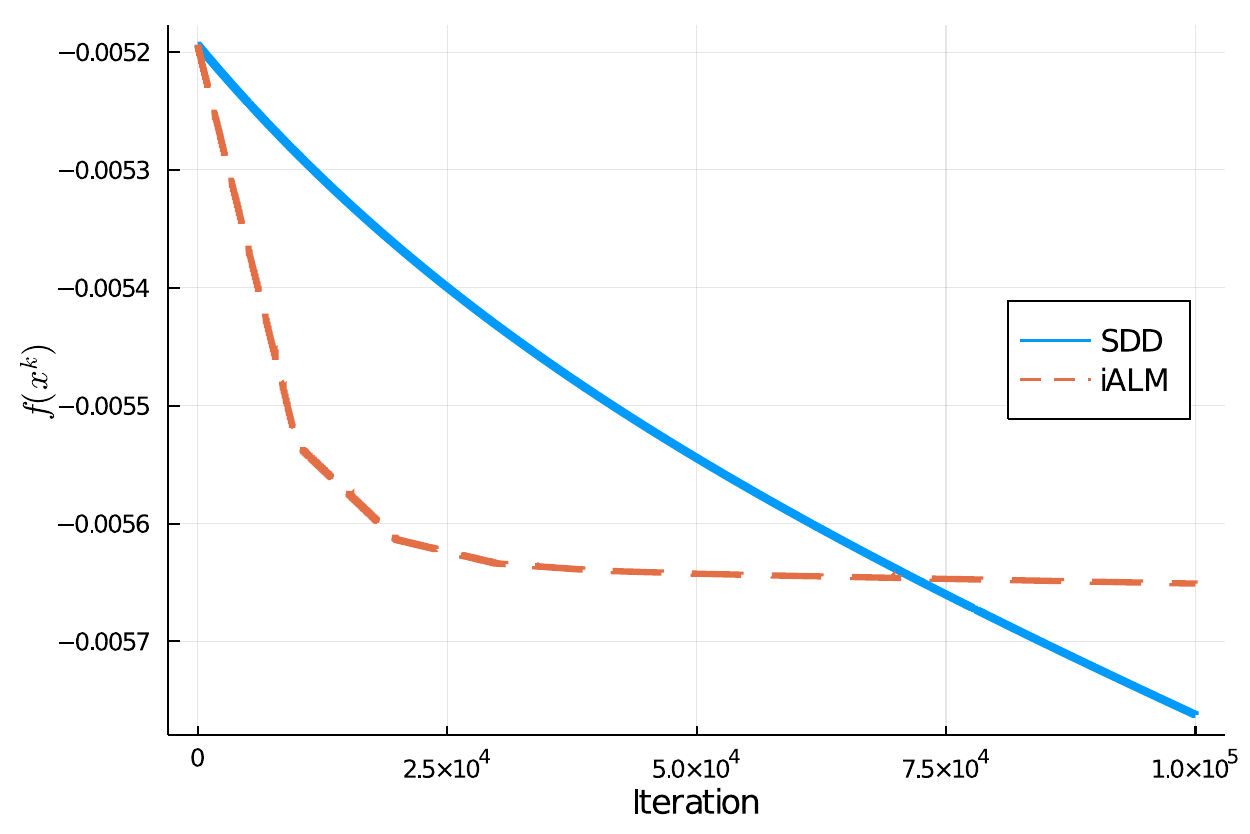} \\[\abovecaptionskip]
	(a) An instance with $n=200$
	\end{tabular}
	\begin{tabular}{@{}c@{}}
	\includegraphics[width=.42\linewidth]{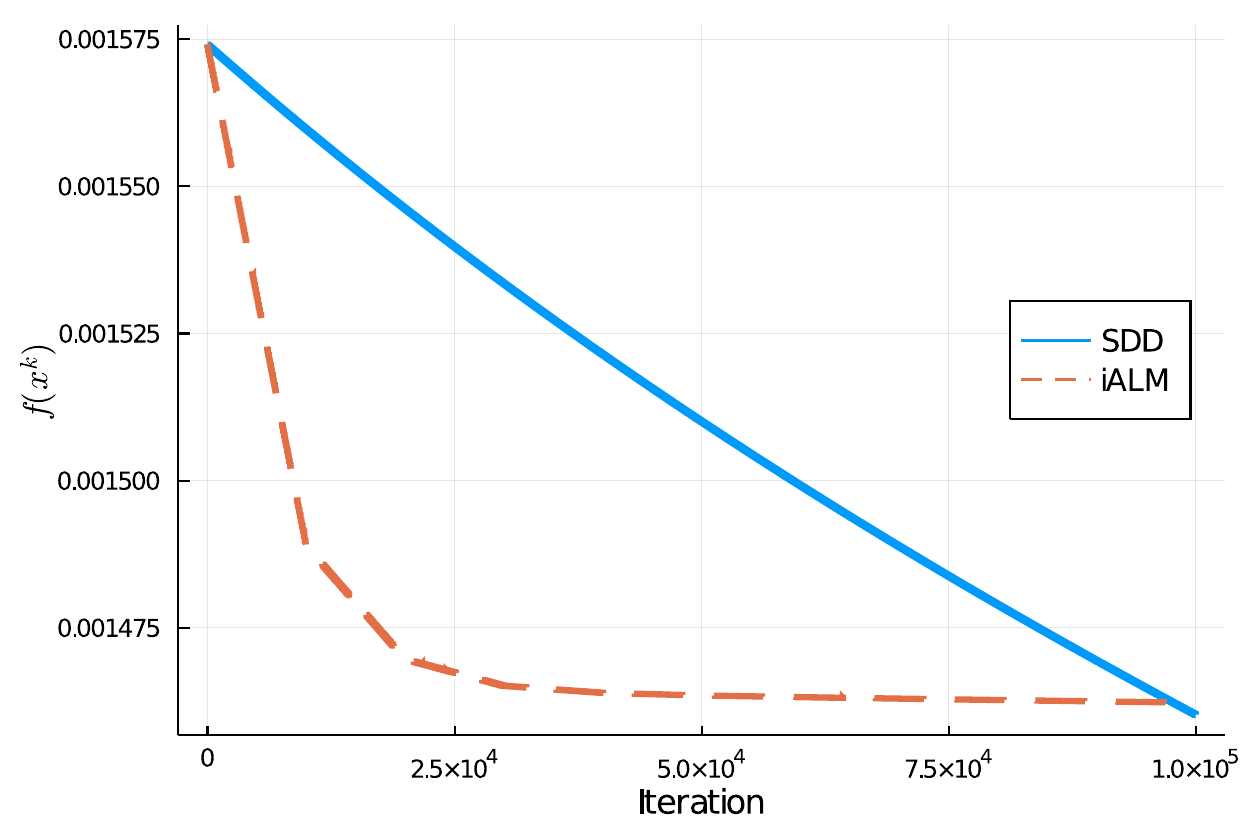} \\[\abovecaptionskip]
	(b) An instance with $n=300$
\end{tabular}}
\caption{Objective Trajectories of SDD-ALM and iALM \cite{li2021rate}} \label{figure: qcqp_obj} \label{fig: compare proximal ALM}
\end{figure*}

\subsection{{SDD-ADMM for Robust Tensor PCA}}
{In this section, we test SDD-ADMM on the robust tensor PCA problem, and compare with the ADMM-g algorithm proposed in \cite{jiang2019structured}. Given a tensor $\mathcal{T} \in \R^{I_1\times I_2 \times I_3}$, the goal is to decompose $\mathcal{T}$ as $\mathcal{Z} + \mathcal{E} + \mathcal{B}$, where $\mathcal{E}$ has a low rank, $\mathcal{E}$ is sparse, and $\mathcal{B}$ represents a small noise. 
% Moreover, when a decomposition with low CP rank is perferred, 
The problem can be casted as the following multiblock problem:
\begin{align}\label{eq: tensor pca}
	\min_{A, B, C, \mathcal{Z}, \mathcal{E}, \mathcal{B}} \{ \|\mathcal{Z}-\llbracket A, B, C\rrbracket \|^2 + \alpha \|\mathcal{E}\|_1 + \alpha_N \|\mathcal{B}\|_F^2 ~|~ \mathcal{Z} + \mathcal{E} + \mathcal{B} = \mathcal{T}\}.
\end{align}
where $A \in \mathbb{R}^{I_{1} \times R}, B \in \mathbb{R}^{I_{2} \times R}, C \in \mathbb{R}^{I_{3} \times R}$, $R$ is an estimate of the true CP-rank,  $\llbracket A, B, C \rrbracket$ denotes the summation of column-wise outer products of $A$, $B$, and $C$, and $\|\cdot\|_F$ denotes the Forbenius norm. See \cite{jiang2019structured} and references therein for a detailed background of the problem. Following standard notations, the Khatri-Rao product, the Hadamard product, and the soft shrinkage operator are denoted by $\odot$, $\circ$, and $\mathbf{S}$, respectively; we use $\mathcal{Z}_{(i)}$ to denote the mode-$i$ unfolding of a tensor $\mathcal{Z}$, and use $\mathbf{0}^{I_1 \times I_2 \times I_3}$ to denote the zero tensor. An adaptive SDD-ADMM tailored to problem \eqref{eq: tensor pca} performs the following updates. 
\begin{algorithm}[H]
	\caption{: \texttt{Adaptive SDD-ADMM} for problem \eqref{eq: tensor pca}} 
	\begin{algorithmic}[1]
		\State \textbf{Input} $\rho,p, \omega, \tau, \gamma,  \overline{\rho}, k_\mathrm{interval} > 0$;
		\State \textbf{Initialize} $A^0, B^0,C^0, \mathcal{Z}^0$ $\mathcal{E}^0$, $\mathcal{B}^0$; $\mu^0 = \mathbf{0}^{I_1\times I_2 \times I_3}$;
		\For{$k=0,1,2,\cdots$}
%		\State set $\eta_k = \theta_1 \mathrm{Lip}(\mu^k,\rho)$;
		\State $A^{k+1} =[(Z)^k_{(1)}(C^k \odot B^k) + 0.5 p A^k][((C^k)^\top C^k)\circ((B^k)^\top B^k) +0.5p I_R]^{-1}$; \label{line4}
		\State $B^{k+1} =[(Z)^k_{(2)}(C^k \odot A^k) + 0.5 p B^k][((C^k)^\top C^k)\circ((A^k)^\top A^k) +0.5p I_R]^{-1}$;s		\State $C^{k+1} =[(Z)^k_{(3)}(B^k \odot A^k) + 0.5 p C^k][((B^k)^\top B^k)\circ((C^k)^\top C^k) +0.5p I_R]^{-1}$;
		\State $\mathcal{E}^{k+1}_{(1)} = \mathbf{S}\left( \frac{\rho}{\rho+p}(\mathcal{T}_{(1)} - \mu^{k}_{(1)}/\rho - \mathcal{B}^k_{(1)} - \mathcal{Z}^{k}_{(1)} ) + \frac{p}{\rho+p} \mathcal{E}^k_{(1)}, \frac{\alpha}{\rho + p} \right)$;
		\State \small{$\mathcal{Z}_{(1)}^{k+1}= \frac{1}{2+2p+\rho}( 2 A^{k+1}\left(C^{k+1} \odot B^{k+1}\right)^{\top}+2p\mathcal{Z}^k_{(1)}-\mu_{(1)}^{k}-\rho\left(\mathcal{E}_{(1)}^{k+1}+\mathcal{B}_{(1)}^{k}-\mathcal{T}_{(1)}\right))$;}
		\State $\mathcal{B}^{k+1}_{(1)} = \frac{1}{\rho + 2\alpha_N + p} \left( p \mathcal{B}^k_{(1)} -\mu^{k}_{(1)} - \rho (\mathcal{Z}^{k+1}_{(1)}  + \mathcal{E}^{k+1}_{(1)} - \mathcal{T}_{(1)}) \right)$; \label{line9}
		\State $\mu^{k+1}_{(1)} = \frac{1}{1+\tau}(\tau\mu^{k+1}_{(1)} - \omega \rho^{-1}(\mathcal{Z}^{k+1}_{(1)} + \mathcal{E}^{k+1}_{(1)} + \mathcal{B}^{k+1}_{(1)} -\mathcal{T}_{(1)}))$;\label{line10}
		\If{$(k+1)\% k_\mathrm{interval} = 0 $}
		\State $\rho \gets  \min(\overline{\rho}, (1+\gamma) \rho )$;
		\EndIf
		\EndFor
	\end{algorithmic}
\end{algorithm}
% We give some comments regarding the above modified adaptive SDD-ADMM. 
Lines \ref{line4}-\ref{line9} are standard proximal ADMM updates, while we update the dual variable in line \ref{line10} via a SDD step. 
Due to the linearity of constraints, each partial proximal Augmented Lagrangian subproblem in SDD-ADMM admit closed-form solutions, and hence we do not necessarily need to perform a proximal gradient step for each block variable. 
% Since each subproblem is optimized exactly, it suffices to set the proximal coefficient to some constant $p>0$, in contrast to $\theta \mathrm{Lip}(\mu^k, \rho)$ as in \eqref{eq: sdd-admm-x}, to establish the descent property of the potential function. 
Moreover, motivated by the adaptive SDD-ADMM in Section \ref{sec: adaptive admm}, we simply increase the penalty $\rho$ by some factor $(1+\gamma)$ after every fixed number of iterations, until some upper bound $\bar{\rho}$ is reached. We also acknowledge that slow convergence of SDD-ADMM is observed with a fixed penalty.}

{Given a parameter $R_{cp} >0$ and dimensions $(I_1, I_2, I_3)$, we guess the initial CP rank by $R = R_{cp} + \lceil 0.2 R_{cp}\rceil$, and generate all data exactly the same way as in \cite{jiang2019structured}. For adaptive SDD-ADMM, we choose $\gamma = 1/3$, $\tau =1/(1+\gamma)$, $\omega = (1+\gamma)/\gamma$, $p = 1$, and the initial penalty $\rho=2$; moreover, we set $k_\mathrm{interval} =10$ and $\bar{\rho}$ = 1e6. For each instance, we run ADMM-g with three different values of $\rho$: $\rho=2$ (the initial penalty passed to adaptive SDD-ADMM), $\rho$ = 1e6 (the maximum penalty used in adaptive SDD-ADMM), and $\rho$ = 1e3 (an intermediate value). For all algorithms, we initialize $(A, B, C)$ with standard Gaussian entries, $\mathcal{Z}$ with the zero tensor, and $(\mathcal{B},\mathcal{E})$ with the tensors used to generate $\mathcal{T}$.}

{
% We test on two different sets of problems: $(I_1, I_2, I_3, R_{cp}) = (10, 20, 30, 10)$ and  $(I_1, I_2, I_3, R_{cp}) = (30, 50, 70, 40)$. 
For $(I_1, I_2, I_3, R_{cp}) = (30, 50, 70, 40)$, we plot $\texttt{res}_k$ and $\texttt{err}_k$ as functions of iteration index $k$ in Fig. \ref{figure5}: here $\texttt{res}_k$ is the geometric mean of the primal residual $\|\mathcal{Z}^{k}+\mathcal{E}^k +\mathcal{B}^k-\mathcal{T}\|_F$ over 3 instances, and $\texttt{err}_k$ is the geometric mean of the relative error $\|\mathcal{Z}^k - \mathcal{Z}^*\|_F/\|\mathcal{Z}^*\|$ over 3 instances, where $\mathcal{Z^*}$ is the low-rank ground truth. The adaptive SDD-ADMM is able to reduce the primal residual close to zero, and recover a $\mathcal{Z}^{k}$ whose relative error is less than $1\%$ or even close to $0.1\%$. In contrast, the performance of ADMM-g is sensitive to the choice of $\rho$: a smaller $\rho$ usually results in a large primal residual, while a larger $\rho$ leads to $\mathcal{Z}^{k}$ with poor quality. Tests on other problem scales exhibit similar behaviors and hence are omitted from presentation.}
%\begin{figure*}[h!]
%\center{
%\begin{tabular}{@{}c@{}}
%	\includegraphics[width=.45\linewidth]{plots/Jul10-1905_res_10-20-30-3.pdf} \\[\abovecaptionskip]
%	(a) Geo. Mean of Primal Residual
%	\end{tabular}
%	\begin{tabular}{@{}c@{}}
%	\includegraphics[width=.45\linewidth]{plots/Jul10-1905_err_10-20-30-3.pdf} \\[\abovecaptionskip]
%	(b) Geo. Mean of Relative Error
%\end{tabular}}
%\caption{$(I_1, I_2, I_3, R_{cp}) = (10, 20, 30, 3)$} \label{figure2}
%\end{figure*}
% \begin{figure*}[h!]
% \center{
% \begin{tabular}{@{}c@{}}
% 	\includegraphics[width=.45\linewidth]{plots/Jul10-1906_res_10-20-30-10.pdf} \\[\abovecaptionskip]
% 	(a) Geo. Mean of Primal Residual
% 	\end{tabular}
% 	\begin{tabular}{@{}c@{}}
% 	\includegraphics[width=.45\linewidth]{plots/Jul10-1906_err_10-20-30-10.pdf} \\[\abovecaptionskip]
% 	(b)  Geo. Mean of Relative Error
% \end{tabular}}
% \caption{$(I_1, I_2, I_3, R_{cp}) = (10, 20, 30, 10)$}\label{figure3}
% \end{figure*}
%\begin{figure*}[h!]
%\center{
%\begin{tabular}{@{}c@{}}
%	\includegraphics[width=.45\linewidth]{plots/Jul10-1908_res_30-50-70-10.pdf} \\[\abovecaptionskip]
%	(a) Geo. Mean of Primal Residual
%	\end{tabular}
%	\begin{tabular}{@{}c@{}}
%	\includegraphics[width=.45\linewidth]{plots/Jul10-1908_err_30-50-70-10.pdf} \\[\abovecaptionskip]
%	(b) Geo. Mean of Relative Error
%\end{tabular}}
%\caption{$(I_1, I_2, I_3, R_{cp}) = (30, 50, 70, 10)$}\label{figure4}
% \end{figure*}
\begin{figure*}[h!]
\center{
\begin{tabular}{@{}c@{}}
	\includegraphics[width=.42\linewidth]{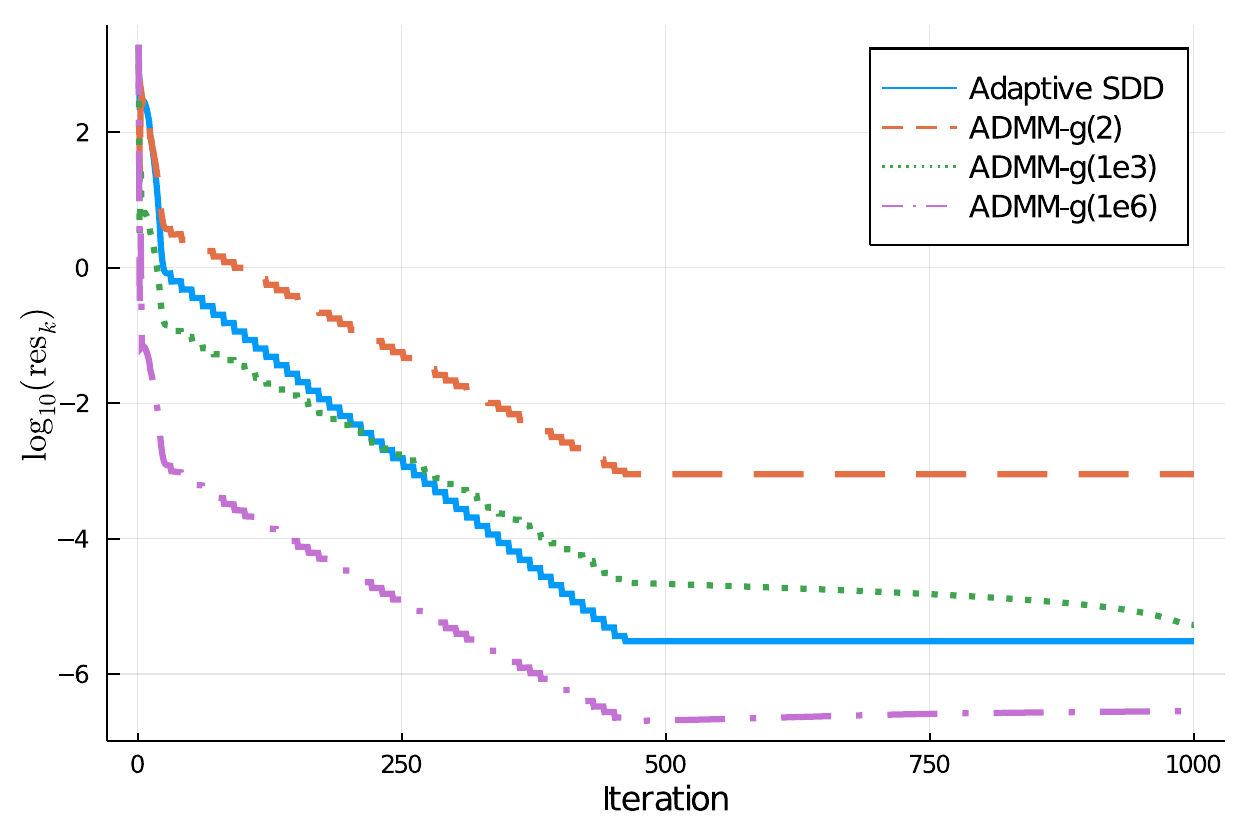} \\[\abovecaptionskip]
	(a) Geo. Mean of Primal Residual
	\end{tabular}
	\begin{tabular}{@{}c@{}}
	\includegraphics[width=.42\linewidth]{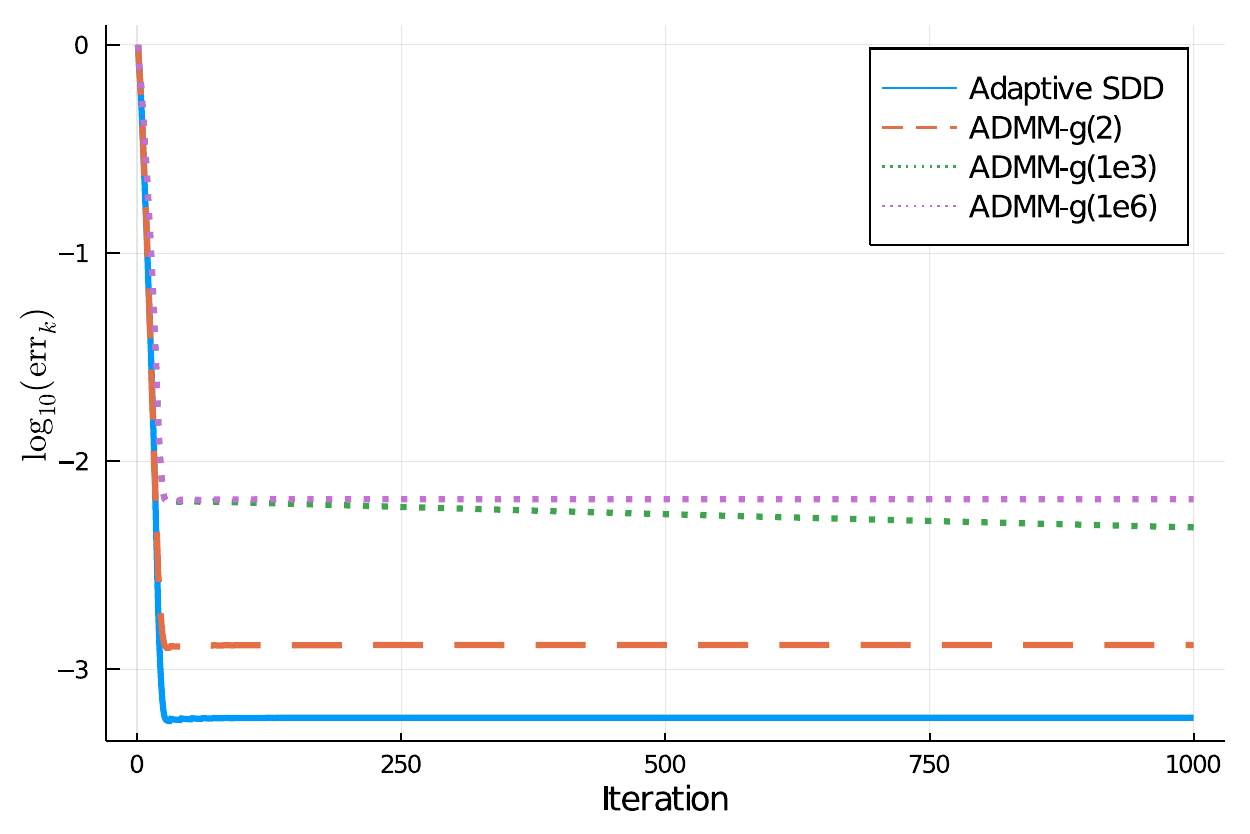} \\[\abovecaptionskip]
	(b) Geo. Mean of Relative Error
\end{tabular}}
\caption{$(I_1, I_2, I_3, R_{cp}) = (30, 50, 70, 40)$}\label{figure5}
\end{figure*}

\subsection{Observations for UDD-ALM \label{section: udd example}}
In this last subsection, we present some experiments on UDD-ALM applied to weakly convex minimization over affine constraints. A key observation in our experiments is that, although the dual step size $\varrho$ is chosen as a constant in our analysis, UDD-ALM is very sensitive to its numerical value. In particular, UDD-ALM may indeed fail to converge when a relatively large $\varrho$ is used, and the order of constraint violation $\|Ax-b\|$ and the order of $\varrho$ are closely related. 
We consider a simple consensus problem 
\begin{align}\label{eq: consensus}
	\min_{x,z\in \R^n } \{f(x) + \alpha \|z\|_1~|~x-z = 0, \|x\|\leq r\},
\end{align}
where $f(x) = -x^\top (U^\top U) x$ and $U \in \R^{n \times n}$ has standard Gaussian entries. We fix $\alpha = r = 1$, $\rho = 1000$, and test UDD-ALM with different dual scaling factors $\texttt{ds}>0$: the dual stepsize is chosen as $\varrho = \rho\times 0.1^{\texttt{ds}}.$ In all runs of UDD-ALM, $x^0$ and $z^0$ are initialized with standard Gaussian entries. The objective values at the end of 2000 iterations are recorded in Table \ref{table2}, and we plot the trajectories of the primal residuals in Fig. \ref{fig7}. We observe that for both instances, UDD-ALM converges to the zero vector as \texttt{ds} gets smaller, which is indeed a stationary point of \eqref{eq: consensus}; moreover, the order of constraint violation drops significantly as well.

 We note that when $\texttt{ds} = 24$, the value of $\varrho$ can be orders-of-magnitude smaller than $\epsilon$ and hence the theoretical $\mathcal{O}(\epsilon^{-2})$ complexity in Theorem \ref{thm: udd convergence} is invalidated.
%  nevertheless, Fig. \ref{fig7} also suggests that the empirical performance of UDD-ALM can be reasonably good for specific problems. 
In particular, when we choose $\varrho$ to be close to zero, UDD becomes the limiting behavior of SDD with $\tau\rightarrow +\infty$, and the resulting algorithm resembles the penalty method, where the dual variables stay close to zero. In other words, the empirical convergence of UDD, to some extent, can be attributed to the penalty method. It is important to note that despite the convergence results presented in Theorem \ref{thm: udd convergence}, we acknowledge that the dual step size $\varrho$ might implicitly affect the fulfillment of the proposed regularity condition at the primal limit point. In practical terms, a large value of $\varrho$ often leads to primal iterates approaching the boundary of $X$, which in turn increases the likelihood of the regularity condition being violated. As a result, we do not claim that UDD-ALM outperforms existing algorithms. Instead, our objective is to share our initial observations on this seemingly counter-intuitive scheme in order to stimulate further exploration and understanding of its potential advantages and limitations.

% Despite the convergence results in Theorem \ref{thm: udd convergence}, we think the dual step size $\varrho$ might implicitly determines if the proposed regularity condition holds at the primal limit point. For example, a large $\varrho$ often encourages primal iterates to hit the boundary of $X$, making the regularity condition more likely to fail. As a result, we do not argue that UDD-ALM outperforms existing algorithms, but rather aim to share some preliminary observations on this seemingly unorthodox scheme in this paper.}

\begin{table}[h!]
\caption{Objective Values Obtained by UDD-ALM}\label{table2}
\begin{center}
\begin{tabular}{lllllll}
\toprule
  $n$   & \texttt{ds}=2 & \texttt{ds}=4  & \texttt{ds}=8  & \texttt{ds}=12  & \texttt{ds}=24  \\
 \hline
% 100	&  17.59  &  3.51e-2  &  2.89e-6 & 2.89e-10 & 2.89e-22    \\
% 200    &   \\	
 500    &   17.59  &  3.51e-2  &  2.89e-6  & 2.89e-10 & 2.89e-22 \\ 
1000    &   25.43  &  6.68e-6  &  6.68e-6  & 6.68e-10 & 6.68e-22 \\  
\bottomrule
\end{tabular}
\end{center}
\end{table}
\begin{figure*}[h!]
\center{
\begin{tabular}{@{}c@{}}
	\includegraphics[width=.42\linewidth]{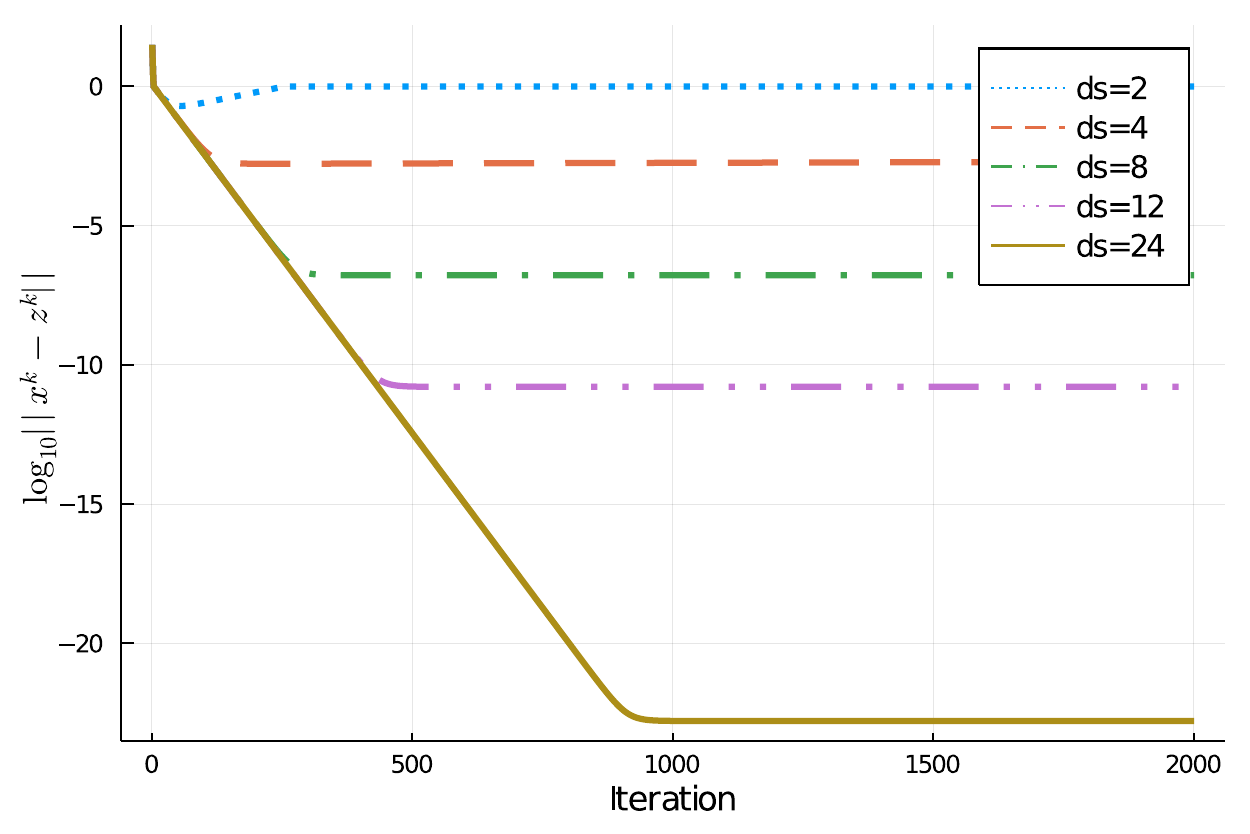} \\[\abovecaptionskip]
	(a) An instance with $n=500$
	\end{tabular}
	\begin{tabular}{@{}c@{}}
	\includegraphics[width=.42\linewidth]{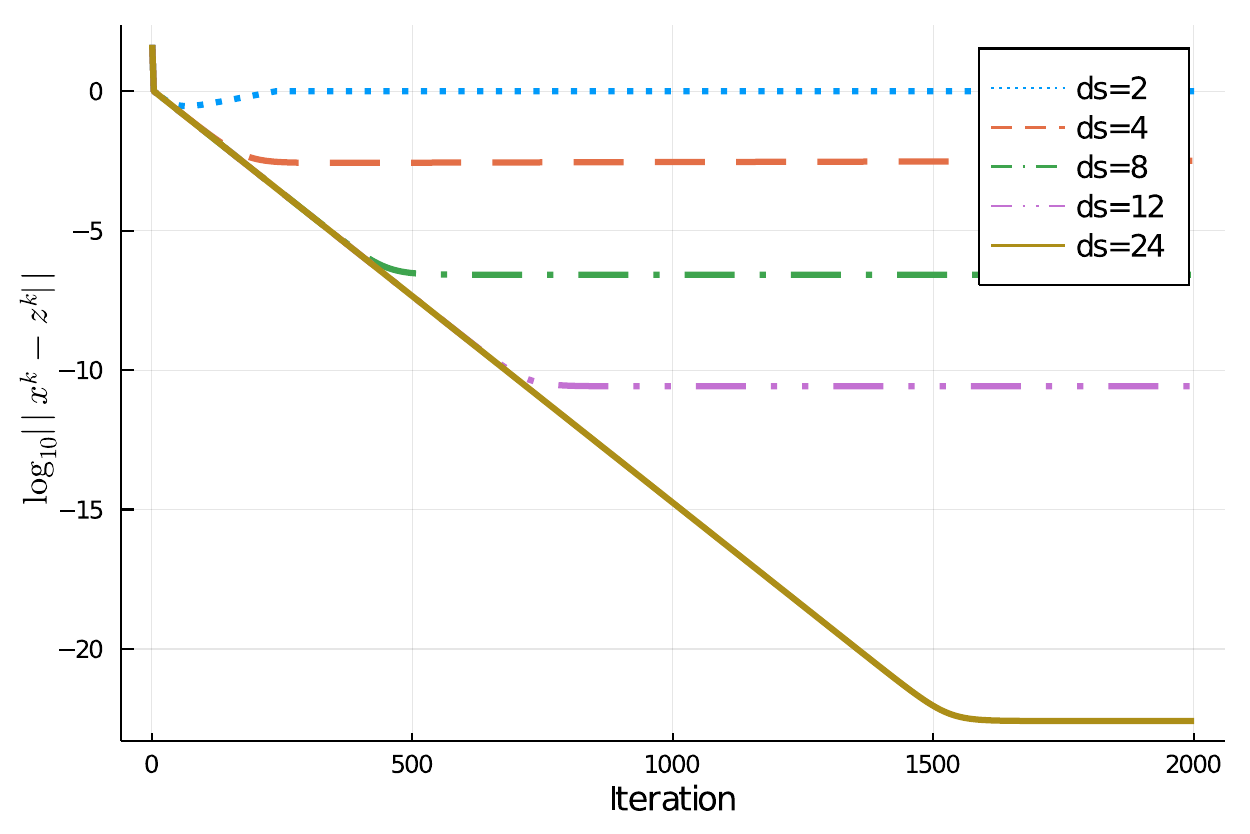} \\[\abovecaptionskip]
	(b) An instance with $n=1000$
\end{tabular}}
\caption{UDD-ALM with Different \texttt{ds}}\label{fig7}
\end{figure*}

\section{Conclusions}\label{sec: conclusion}
This paper proposes two new algorithms based on the concept of dual descent: SDD-ADMM and UDD-ALM. We apply SDD-ADMM to solve nonlinear equality-constrained multi-block problems, and establish an $\mathcal{O}(\epsilon^{-4})$ iteration complexity upper bound, or $\mathcal{O}(\epsilon^{-3})$ and  $\mathcal{O}(\epsilon^{-2})$ under additional technical assumptions. When UDD-ALM is applied for weakly convex minimization over affine constraints, we show that under a regularity condition, the algorithm asymptotically converges to a stationary point and finds an approximate solution in $\mathcal{O}(\epsilon^{-2})$ iterations. Our iteration complexities for both algorithms either achieve or improve the  best-known results in the ADMM and ALM literature. Moreover, SDD-ADMM addresses a long-standing limitation of existing ADMM frameworks. 

Nevertheless, the behavior of UDD-ALM is somehow not fully understood. Theoretically the dual stepsize $\varrho$ is treated as a constant, while, as we illustrate numerically, the convergence of UDD-ALM can be very sensitive to its numerical value. We conjecture that the modified Robinson's condition required on the limit point can be implicitly affected by the dual stepsize. This issue seems to be highly problem-dependent, and we leave it as our future work.
\appendix

\section{Examples Satisfying Assumption \ref{assumption: improved rate}}\label{examples}
We give some examples where Assumption \ref{assumption: improved rate} can be satisfied. Suppose $i=p=1$ for simplicity. We assume that 
for all $x\in X(r)$, $\nabla h(x)$ has full column rank, and their smallest singular values are bounded away from zero, i.e.,  $\sigma(r): = \inf_{x\in X(r)} \sigma_{\min}(\nabla h(x)) > 0. $
% \begin{align*}
% 	\sigma(r): = \inf_{x\in X(r)} \sigma_{\min}(\nabla h(x)) > 0. 	
% \end{align*}
\begin{example}\label{example1}
	Let $g$ be a possibly nonconvex function with 
	\begin{align*}
		M_g: = \sup \{\|\xi_g\|: \xi_g \in \partial g(x), x\in X(r)\} < +\infty,
	\end{align*}
	i.e., $g$ is Lipschitz over $X(r)$.  Then we have
	\begin{align*}
		\mathrm{dist}(-\nabla h(x) \mu, \partial g(x)) 
         =&  \inf_{\xi_g \in \partial g(x)} \|\nabla h(x)\mu + \xi_g\|\\
		 \geq & \|\nabla h(x) \mu\| + \inf_{\xi_g\in \partial g(x)} -\|\xi_g\| \geq \sigma \|\mu\| - M_g.	
	\end{align*}
	Hence Assumption \ref{assumption: improved rate} is satisfied with $\sigma = \sigma(r)$. 
\end{example}
\begin{example}\label{example2}
	Let $g = \delta_{X} + \tilde{g}$ where $X$ is a full-dimensional compact convex set and $\tilde{g}$ is a convex function over $X$. By \cite[Theorem 24.7 and Theorem 23.8]{rockafellar1970convex},
	\begin{align*}
		& M_g:= \sup \{ \| \xi_{\tilde{g}}\|: \ \xi_{\tilde{g}} \in \partial \tilde{g}(x), x\in X  \} < + \infty, \text{~and~}\partial g(x) = \partial \tilde{g}(x) + N_{X}(x), \ \forall x\in X,
	\end{align*}
	where $N_{X}(x)$ denotes the normal cone of $X$ at $x\in X$. Further assume that $X(r)$ belongs to $\mathrm{int}~X$, the interior of $X$, so that $N_{X}(x) = \{0\}$ for all $x\in X(r)$. As a result, 
	\begin{align*}
		\mathrm{dist}(-\nabla h(x) \mu, \partial g(x)) =&  \inf\{ \|\nabla h(x)\mu + \xi_{\tilde{g}}+ d_g\|: \  \xi_{\tilde{g}}\in \partial \tilde{g}(x), d_g\in N_{X}(x)\}	\\
		\geq &  \|\nabla h(x)\mu \|	 - M_g\geq \sigma(r) \|\mu\| - M_g.
	\end{align*}
	So again Assumption \ref{assumption: improved rate} is satisfied with $\sigma = \sigma(r)$. 
\end{example}
\begin{example}\label{example3}
	Suppose $g = \delta_X$ and $X := \{x\in \R^n~|~F(x)\in  D\}$, where $F:\R^n \rightarrow \R^p$ are continuously differentiable and $D\subset \R^p$ $(p \leq  n-m)$. 
	Further suppose that for any $x\in X(r)$, the Jacobian matrix $J(x): = [\nabla h(x), \nabla F(x)] \in \R^{n \times (m+p)}$ has full column rank, and $\sigma : = \min_{x\in X(r)} \sigma_{\min}(J(x)) > 0.$
	Then by \cite[Theorem 6.14]{rockafellar2009variational}, for all $x\in X(r)$, it holds that 
%	$\partial g(x) = N_X(x) \subset \{ \nabla F(x) y~|~y\in N_D(F(x))\}$.
 	\begin{align*}
 		\partial g(x) = N_X(x) \subset \{ \nabla F(x) y~|~y\in N_D(F(x))\}.	
 	\end{align*}
	Denote $u = [\mu^\top, y^\top]^\top$; since  $\|J(x) u\|\geq \sigma \|u\|\geq \sigma \|\mu\|$, we have 
	\begin{align*}
		\mathrm{dist}(-\nabla h(x) \mu, \partial g(x)) \geq &  \inf\{\| \nabla h(x) \mu + \nabla F(x)y\| : y\in N_D(F(x))\} \\
		= &  \inf\{\| J(x)u\| : y\in N_D(F(x))\}  \geq   \sigma \|\mu\|,
	\end{align*}
	In particular, consider $h(x) = Ax-b$, and $X = \{x\in \R^n~|~l\leq Cx \leq u \}$ where $C\in \R^{p \times n}$ and $l, u\in \R^p$. Then Assumption \ref{assumption: improved rate} holds as long as rows of $A$ and $C$ are linearly independent. 
\end{example}

\section{UDD-ALM with Nonlinear Constraints}\label{sec: udd-nonlinear}
In this section we apply UDD -ALM to deal with nonlinear constraints and establish its convergence by assuming a descent solution oracle of each augmented Lagrangian relaxation.

\begin{assumption}\label{assumption: udd-nonlinear}
	We make the following assumptions regarding problem \eqref{eq: lc problem}.
	\begin{enumerate}
		\item The function $g:\R^{n} \rightarrow \overline{\R}$ can be decomposed as $\tilde{g} + \delta_{X}$, where $X \subseteq \R^n$ is compact and described by a finite number of inequality constraints, i.e., $X = \{x\in \R^n|q_l(x) \leq 0, \forall l \in [L]\}$ where $q_l:\R^n \rightarrow \R$ is continuously differentiable for $l\in [L]$, and $\tilde{g}:\R^{n} \rightarrow \R$ is continuous and convex over $X$. 
		\item The function $f:\R^n \rightarrow \R$ is continuously differentiable over $X$.
		\item The constraints $h:\R^n \rightarrow \R^m$ are given by $h(x) = [h_1(x) , \cdots, h_m(x)]^\top$, where $h_j:\R^n \rightarrow \R$ is continuously differentiable over $X$ for $j\in [m]$.
	\end{enumerate}
\end{assumption}
Denote $\nabla h(x) = [\nabla h_1(x), \cdots, \nabla h_m(x)]\in \R^{n \times m}$ in this subsection. We also define an approximate KKT point for problem \eqref{eq: lc problem} under Assumption \ref{assumption: udd-nonlinear} as follows.
\begin{definition}
	Let $\epsilon > 0$. We say $x$ is an $\epsilon$-KKT point for problem \eqref{eq: lc problem} if 
	\begin{subequations}
		\begin{align}
			& \mathrm{dist}\Big( -\nabla f(x) - \nabla h(x) \mu - \sum_{l=1}^L \nabla q_l(x)y_l, \partial \tilde{g}(x) \Big) \leq \epsilon, \ \|h(x)\| \leq \epsilon, 	\label{eq: kkt1-def} \\
			& q_l(x) \leq 0, \ y_l g_l(x) = 0, \ \forall l\in [L].	\label{eq: kkt2-def}
		\end{align}
	\end{subequations}
	for some $\mu \in \R^m$ and $y \in \R^L_+$. We simply say $x$ is a KKT point when $\epsilon=0$.
\end{definition}

% Recall the augmented Lagrangian function $$L_{\rho}(x, \mu) = f(x) + \tilde{g}(x) + \langle \mu, h(x)\rangle + \frac{\rho}{2}\|h(x)\|^2.$$  
The UDD-ALM with nonlinear constraints is almost the same as Algorithm \ref{alg: UDD-ALM}, except that we replace the primal update \eqref{eq: udd-alm-x} by the following nonlinear program:
\begin{align}\label{eq: udd-alm-x-nonlinear}
	\min_{x\in X}L_{\rho}(x, \mu^k) + \frac{c}{2}\|x-x^k\|^2,
\end{align}
where $L_{\rho}(x, \mu) = f(x) + \tilde{g}(x) + \langle \mu, h(x)\rangle + \frac{\rho}{2}\|h(x)\|^2.$
% We present the UDD-ALM with nonlinear constraints in Algorithm \ref{alg: UDD-ALM-Nonlinear}.
% \begin{algorithm}[H]
% 	\caption{: Unscaled Dual Descent ALM (UDD-ALM) with Nonlinear Constraints} \label{alg: UDD-ALM-Nonlinear}
% 	\begin{algorithmic}[1]
%         \State \textbf{Initialize} $x^0\in X$, $\mu^0 \in \R^m$, $\rho > 0$, $\varrho>0$, and $c >0$;
% 		\For{$k=0,1,2\cdots$}
% 		\State update $x^{k+1}$ by solving the following augmented Lagrangian relaxation:
% 		\begin{align}\label{eq: udd-alm-x-nonlinear}
% 			\min_{x\in X}L_{\rho}(x, \mu^k) + \frac{c}{2}\|x-x^k\|^2;
% 		\end{align}
% 		\State update $\mu^{k+1}$ through an unscaled dual descent update: 
% 		\begin{align}\label{eq: udd-alm-dual-nonlinear}
% 			\mu^{k+1} = \mu^k -  \varrho h(x^{k+1});
% 		\end{align}
% 		\EndFor
% 	\end{algorithmic}
% \end{algorithm}
Next we define a descent solution oracle for problem \eqref{eq: udd-alm-x-nonlinear}. 
\begin{assumption}\label{assumption: descent oracle}
	Given $x^k \in X$ and $\mu^k \in \R^m$, we can find $x^{k+1}$ such that
	\begin{align}\label{eq: sufficient descent}
		L_{\rho}(x^{k+1}, \mu^k) \leq L_{\rho}(x^{k}, \mu^k) - \nu \|x^{k+1}-x^k\|^2 	
	\end{align}
	for some $\nu >0$, and there exists $y^{k+1} \in \R^L_{+}$ such that
	\begin{subequations}\label{eq: kkt}
		\begin{align}
			& 0 \in  \partial_x L_{\rho}(x^{k+1}, \mu^k) + c (x^{k+1}-x^k) + \sum_{l=1}^L\nabla q_l(x^{k+1})y^{k+1}_l, \label{eq: kkt-1}\\
			 & q_l(x^{k+1}) \leq 0, \ y^{k+1}_l q_l(x^{k+1}) = 0, \ \forall l \in [L].\label{eq: kkt-2}
		\end{align}
	\end{subequations}
\end{assumption}
\begin{remark}
	Assumption \ref{assumption: descent oracle} 	requires $x^{k+1}$ to be a KKT point of problem \eqref{eq: udd-alm-x-nonlinear} with an improved objective value compared to the previous iterate $x^{k}$. Notice that the sufficient descent condition \eqref{eq: sufficient descent} can be satisfied with $\nu = c/2$ if some global solver for problem \eqref{eq: udd-alm-x-nonlinear} is available. To this end, we also note that given $x^k\in X$ and $\mu^k\in \R^m$, problem \eqref{eq: udd-alm-x-nonlinear} is convex if $X$ is convex, $f$ and $h_1,\cdots, h_m$ have continuous Hessians over $X$, and $c$ is sufficiently large. We adopt \eqref{eq: kkt} to avoid unnecessary technicality, while it is possible to allow $x^{k+1}$ to be an inexact KKT solution of \eqref{eq: udd-alm-x-nonlinear}. 
\end{remark}

Since $X$ is assumed to be compact, the sequence $\{x^k\}_{k\in \N}$ has at least one limit point $x^*\in X$. The next lemma shows that if $x^*$ satisfies the linearly independence constraint qualification (LICQ), then $\{\mu^{k}\}_{k\in \N}$ has a bounded subsequence.
\begin{lemma}\label{lemma: nonlinear-udd-dual-bound}
	Suppose Assumptions \ref{assumption: udd-nonlinear} and \ref{assumption: descent oracle} hold. Let $x^*\in X$ be a limit point of $\{x^k\}_{k\in \N}$ generated by UDD-ALM, and $\{x^{k_r}\}_{r\in \N}$ be the corresponding convergent subsequence.
	Denote $I(x^*) = \{l \in [L]~|~ q_l(x^*) =0\}$. Suppose that the matrix
	\begin{equation}\label{eq: jacobian_LICQ}
	H^*:= [\nabla h_1(x^*),\cdots, \nabla h_m(x^*), \{\nabla q_l(x^*)\}_{l\in I(x^*)}] \in \R^{n \times (m+|I(x^*)|)}
	\end{equation}
	has full column rank, then the sequence $\{\mu^{k_r}\}_{r\in \N}$ is bounded.
\end{lemma}
\begin{proof}
	For $l\notin I(x^*)$, we have $g_l(x^{k_r}) < 0$ and thus $y^{k_r}_l = 0$ by \eqref{eq: kkt-2} for all sufficiently large $r \in \N$. Hence, \eqref{eq: kkt-1} becomes 
	$H_{k_r} [(\mu^{k_r})^\top, (y^{k_r}_{I(x^*)})^\top]^\top= e_{k_r},$
	where 
	\begin{align*}
		H_{k_r} := & [\nabla h_1(x^{k_r}),\cdots, \nabla h_m(x^{k_r}), \{\nabla q_l(x^{k_r})\}_{l\in I(x^*)}], \\
		e_{k_r} :=& -\nabla f(x^{k_r}) - \xi^{k_r}_{\tilde{g}} - (\rho+\varrho)\nabla h(x^{k_r}) h(x^{k_r}) - c(x^{k_r}-x^{k_r-1}),
	\end{align*}
	$y^{k_r}_{I(x^*)} \in \R^{|I(x^*)|}_+$ is the sub-vector of $y^{k_r}$ specified by indices in $I(x^*)$, and $\xi^{k_r}_{\tilde{g}} \in \partial \tilde{g}(x^{k_r})$. Since $H^*$ has full column rank, so does $H^{k_r}$ for sufficiently large $r\in \N$, which suggests that $\|\mu^{k_r}\|\leq \| \mu^{k_r}\|+\| y^{k_r}_{I(x^*)}\| \leq \| (H_{k_r}^{\top}H_{k_r} )^{-1}H_{k_r}^{\top} \| \| e_{k_r}\|$.
% 	\begin{align}\label{eq: nonlinear udd dual bound}
% %			\begin{bmatrix} \mu^{k_r} \\ y^{k_r}_{I(x^*)}\end{bmatrix}  = \left(H_{k_r}^{\top}H_{k_r} \right)^{-1}H_{k_r}^{\top} e_{k_r}.
% 	\|\mu^{k_r}\|\leq \| \mu^{k_r}\|+\| y^{k_r}_{I(x^*)}\| \leq \left\| \left(H_{k_r}^{\top}H_{k_r} \right)^{-1}H_{k_r}^{\top} \right\| \| e_{k_r}\|.
% 	\end{align}
	Due to the compactness of $X$, the continuity of $\tilde{g}$, and the continuous differentiability of $f$, $g_l$'s, and $h_j$'s, we know that {$\|e_{k_r}\|$} is bounded by some finite constant depending on the problem data $(f, \tilde{g}, X, h)$ as well as parameters $(\rho, \varrho, c)$. {As a result of the previous inequality}, the sequence $\{\mu^{k_r}\}_{r\in \N}$ is bounded.
\end{proof}

% Now we present the convergence of UDD-ALM with nonlinear constraints.
\begin{theorem}\label{thm: nonlinear udd}
	Suppose Assumptions \ref{assumption: udd-nonlinear}-\ref{assumption: descent oracle} hold. Let $x^*\in X$ be a limit point of $\{x^k\}_{k\in \N}$ generated by UDD-ALM that satisfies the LICQ condition, i.e., $H^*$ defined in \eqref{eq: jacobian_LICQ} has full column rank. Then the following statements hold.
	\begin{enumerate}
		\item (Asymptotic Convergence) The point $x^*$ is a KKT point of problem \eqref{eq: lc problem}.
		\item (Iteration Complexity) Let $\epsilon >0$. Define constants $\sigma_1: = \min \{\nu, \varrho\}$ and $\sigma_2:=	c+ (\rho+\varrho) \max_{x\in X}\|\nabla h(x)\|.$
		% \begin{align*}
		% 	\sigma_1: = \min \{\nu, \varrho\}, \ \sigma_2:=	c+ (\rho+\varrho) \max_{x\in X}\|\nabla h(x)\|.
		% \end{align*}
		UDD-ALM finds an $\epsilon$-KKT point in at most $K$ iterations, where
		\begin{align} 
			K \leq \left\lceil \frac{\max\{1, \sigma_2\}^2(L_{\rho}(x^0, \mu^0) - f(x^*)-\tilde{g}(x^*))}{\sigma_1 \epsilon^2} \right\rceil = \mathcal{O}(\epsilon^{-2}).\notag %\label{eq: nonlinear udd iter complex}
		\end{align}	
	\end{enumerate}
\end{theorem}
\begin{proof}
	First note that $L_{\rho}(x^{k}, \mu^k) - L_{\rho}(x^{k+1}, \mu^{k+1}) \geq \nu \|x^{k+1}-x^k\|^2 + \varrho\|h(x^{k+1})\|^2$ for all $k\in \N$ by \eqref{eq: sufficient descent} and the UDD update.
	% \begin{align*}
	% 	L_{\rho}(x^{k}, \mu^k) - L_{\rho}(x^{k+1}, \mu^{k+1}) \geq \nu \|x^{k+1}-x^k\|^2 + \varrho\|h(x^{k+1})\|^2.
	% \end{align*}
	Then with the help of Lemma \ref{lemma: nonlinear-udd-dual-bound}, the claims can be proved via straightforward modification of the proof of Theorem \ref{thm: udd convergence}.
\end{proof}
{
\begin{remark}
	We note that the $\mathcal{O}(\epsilon^{-2})$ complexity bound in Theorem \ref{thm: udd convergence} is measured by first-order oracles of the problem data, whereas the iteration complexity in Theorem \ref{thm: nonlinear udd} is measured by the subproblem oracle defined in Assumption \ref{assumption: descent oracle}. Information including $\|\mu^k\|$ and $\rho$ may affect the computation effort to evaluate the subproblem oracle, which we do not consider explicitly in Theorem \ref{thm: nonlinear udd}. 
\end{remark}}

%\section*{Acknowledgments}

\bibliographystyle{siamplain}
\bibliography{ref}
\end{document}